%% file: main.tex
\documentclass[11pt]{article}
\usepackage[margin=1.05in]{geometry}

\usepackage[english]{babel}
\usepackage{mathtools,amssymb,amsthm,mathrsfs}
\usepackage[hidelinks]{hyperref}

\title{\LARGE A Metric with Positive Sectional Curvature on $S^2\times S^3$}
\author{
  Shengtao Guo \qquad
  Ethan X. Fang \qquad
  Junwei Lu\thanks{Department of Biostatistics, Harvard T.H. Chan School of
	Public Health. Email: \texttt{junweilu@hsph.harvard.edu}.}
}
\date{}

\newtheorem{theorem}{Theorem}[section]
\newtheorem{proposition}[theorem]{Proposition}
\newtheorem{lemma}[theorem]{Lemma}

\theoremstyle{definition}

\theoremstyle{remark}

\newcommand{\R}{\mathbb R}
\newcommand{\Z}{\mathbb Z}
\newcommand{\cM}{\mathcal M}
\newcommand{\cZ}{\mathcal Z}
\newcommand{\tZ}{\widetilde{\mathcal Z}}
\newcommand{\Gr}{\operatorname{Gr}}
\newcommand{\Rm}{\operatorname{Rm}}
\newcommand{\II}{\operatorname{II}}
\newcommand{\Span}{\operatorname{span}}
\newcommand{\dist}{\operatorname{dist}}
\newcommand{\Id}{\operatorname{Id}}
\newcommand{\Lie}{\mathcal L}
\newcommand{\dd}{\mathrm d}
\newcommand{\eps}{\varepsilon}

\allowdisplaybreaks[2]

\begin{document}

\maketitle

\begin{abstract}
We prove that $S^2\times S^3$ admits a Riemannian metric with positive
sectional curvature.  We view it as a principal circle bundle over
$S^2\times S^2$.  A diagonal Cheeger deformation of the base and a connection
whose curvature form vanishes on the remaining flat tori yield a
nonnegatively curved connection metric whose zero-curvature planes are the
horizontal lifts of the tangent planes to those tori.  We then perturb this
metric by the real part of a global complex-valued symmetric $2$-tensor.
Differentiation along the circle fibers produces a trace-free first variation
of the second fundamental form on local horizontal lifts of the flat tori.
The Gauss equation converts this into a positive second-order curvature term
that dominates as the fibers shrink.  A quantitative lower bound for the
Hessian in directions normal to the set of zero-curvature planes extends
this positivity to nearby planes. The metric and the proof are discovered by the Odin Automatic AI Research Agent.
\end{abstract}

\input{sections/introduction}
\input{sections/preliminaries}
\input{sections/construction}
\input{sections/cheeger}
\input{sections/connection}
\input{sections/perturbation-tensor}
\input{sections/perturbation}
\appendix
\input{sections/technical-appendix}
\input{sections/topology}

\bibliographystyle{alpha}
\bibliography{references}

\end{document}

%% file: sections/introduction.tex
\section{Introduction}

The construction of metrics with positive sectional curvature is one of the
central existence problems in Riemannian geometry.  In contrast with the
flexibility of positive Ricci or scalar curvature, only a small collection of
general mechanisms for producing positive sectional curvature is known; see,
for example, the surveys of Grove~\cite{Grove2009} and
Ziller~\cite{Ziller2007}.  Products of spheres have long served as basic test
cases.  Their product metrics have nonnegative curvature and many flat mixed
planes, but removing all of those planes by a global metric deformation is a
delicate problem.

Berger~\cite{Berger1966} proved that if a metric variation of a compact Riemannian product has
nonnegative first variation of sectional curvature on every mixed plane,
then that first variation vanishes on every mixed plane.
Bourguignon, Deschamps, and Sentenac~\cite{BDS1972} developed higher-order variational
extensions of this result, and
Bourguignon~\cite{Bourguignon1975} subsequently applied related ideas to
Hopf's conjecture on Riemannian products.  Cheeger~\cite{Cheeger1973} introduced the deformation
by group actions that now bears his name.  M\"uter~\cite{Muter1987} gave the systematic analysis of these
deformations used below.  Ziller~\cite{Ziller2009} summarized M\"uter's results, including the
description of the planes that remain flat and a reparametrized curvature
formula with a nonnegative remainder.  For the
diagonal $\mathrm{SO}(3)$-action on $S^2\times S^2$, the resulting
Cheeger-deformed metrics are nonnegatively curved, and their zero-curvature
planes are tangent to a family of totally geodesic flat tori; see
\cite{Muter1987} and \cite[Proposition~2.3]{Bettiol2014}.
Strake's analysis of first variations~\cite{Strake1987} explains an obstruction
created by such tori: no metric variation can increase to first order the
sectional curvature of every plane tangent to them.  Related obstructions to
first-order improvement were established by Spatzier and Strake
\cite[Section~5]{SpatzierStrake1990}.  Bettiol nevertheless used a local
conformal variation to construct metrics with the following property: at
every point, the average sectional curvature of any two $2$-planes in the
same tangent space separated by at least a prescribed Grassmannian distance
is positive.  The prescribed distance may be chosen arbitrarily small.  In
particular, these metrics have positive
biorthogonal curvature~\cite{Bettiol2014}.

There was also substantial progress on $S^2\times S^3$ under weaker
positivity conditions.  Wilking constructed an almost positively curved
metric on $\mathbb{RP}^2\times\mathbb{RP}^3$; its pullback gives such a
metric on $S^2\times S^3$~\cite[Corollary~3 and
Proposition~6]{Wilking2002}.  Adapting Bettiol's first-order conformal
deformation to Wilking's metric, Stupovski and Torres obtained metrics on
$S^2\times S^3$ for which the average sectional curvature of any two planes
separated by at least a prescribed Grassmannian distance is positive.  In particular,
their metrics have positive biorthogonal
curvature~\cite[Theorem~A]{StupovskiTorres2020}.
These averaged curvature conditions are weaker than positive sectional
curvature and therefore do not imply the theorem below.

In a recent preprint, Brendle and Hung~\cite{BrendleHung2026} prove that
$S^2\times S^2$ itself admits a metric with positive sectional curvature.
They start from the background metric obtained by the same diagonal Cheeger
deformation and construct a
third-order deformation.  Their argument includes an analysis of
the minimum of sectional curvature near the set of zero-curvature planes
\cite[Theorems~2.11 and~2.12]{BrendleHung2026} and a quantitative lower bound
for the sectional curvature of the Cheeger-deformed metric
\cite[Proposition~B.1 and Corollary~B.2]{BrendleHung2026}.  Motivated by their
work, we prove the following theorem.

\begin{theorem}                                                                    \label{thm:main}
The manifold $S^2\times S^3$ admits a Riemannian metric with
strictly positive sectional curvature.
\end{theorem}

We realize $S^2\times S^3$ as the principal $S^1$-bundle with primitive
first Chern class $(1,1)$ over $S^2\times S^2$.  The metric is constructed in
three stages.  First, we make a small diagonal Cheeger deformation of the
product metric on the base.  Its zero-curvature planes form a smooth compact
embedded four-dimensional submanifold of the Grassmann bundle.  We prove a quantitative
lower bound with explicit dependence on the Cheeger parameter, including
at points $(p_1,p_2)\in S^2\times S^2$ where $p_1=\pm p_2$.

Second, we choose an explicit representative of the Chern class.  If
$\omega_1$ and $\omega_2$ are the standard area forms, this representative
has the form $\Omega_t=\tfrac12(\omega_1+\omega_2+t\,\dd\alpha)$, where
$\alpha$ is an explicit one-form invariant under the diagonal
$\mathrm{SO}(3)$-action.  The form $\Omega_t$ is symplectic and vanishes on
every torus in this family.  The
corresponding connection metric on the circle bundle is nonnegatively curved,
and its zero-curvature planes are precisely the horizontal lifts of the base
zero-curvature planes.  Completing the square in the curvature formula for
the connection metric yields a condition involving the curvature form and
its covariant derivative.  For principal circle bundles, this condition
incorporates Weinstein's fatness criterion~\cite{Weinstein1980} and is a
special case of criteria proved by Chaves, Derdzinski, and
Rigas~\cite[main theorem, condition~(ii)]{ChavesDerdzinskiRigas1992} and by
Shankar, Tapp, and Tuschmann
\cite[Theorem~3.1 and Lemma~3.2]{ShankarTappTuschmann2005}.  Strict versions of
these inequalities are discussed in \cite[Section~6]{Ziller2007}.  Related
curvature inequalities for metrics on vector and sphere bundles were developed by Strake and
Walschap~\cite{StrakeWalschap1990} and Tapp~\cite{Tapp2003}.  We apply the
circle-bundle criterion when the base has nonnegative sectional curvature
and contains zero-curvature planes.  On mixed planes, the derivative at
$t=0$ of $\nabla^{g_t}\Omega_t$ vanishes.  This cancellation ensures that
equality can occur only on those tori.

Third, we perturb the connection metric by the real part of a globally
defined complex-valued symmetric $2$-tensor.  Under the principal action of
$e^{i\lambda}\in S^1$, this tensor is multiplied by $e^{2i\lambda}$.  On every
zero-curvature plane, the perturbation is nonzero and trace-free.  On each local
horizontal lift of one of the flat tori, the induced metric remains flat
under the perturbation.  Such a lifted surface is totally geodesic for the
background metric.  Differentiating the Gauss equation therefore shows that
the first variation of ambient sectional curvature vanishes on its tangent
planes, in accordance with Strake's obstruction to such a first-order
improvement~\cite[Lemma~4.1]{Strake1987}.  At second order, differentiation in
the fiber direction produces a large trace-free second fundamental form, and
the Gauss equation converts it into a positive curvature contribution.  When
the fibers are sufficiently short, this contribution dominates the errors
arising from nearby planes.

This last step is where our construction differs most from that of Brendle and
Hung~\cite{BrendleHung2026}.  They deform the four-dimensional base
through third order; their second-order lower bound remains degenerate along
a two-torus, and a third-order correction removes the remaining degeneracy
\cite[Propositions~5.2 and~6.1]{BrendleHung2026}.  We
instead keep the base nonnegative, pass to its circle bundle, and use the
additional vertical direction to obtain positivity already at second order
from a single linear metric perturbation.  The bundle argument requires two
additional ingredients.  The first is a global complex-valued symmetric
$2$-tensor satisfying the transformation law above and remaining smooth where
$p_1=\pm p_2$.  The second is an estimate for planes with a vertical component
that gives the precise $\eps$-dependence in the horizontal and vertical normal
directions.  Here $\eps$ measures the length of the circle fibers.
The curvature of planes containing the fiber direction is of order $\eps^2$.
The estimate shows that the negative correction produced by minimizing over
nearby planes remains bounded, while the positive term from
the Gauss equation grows like $\eps^{-2}$.

\vspace{3pt}

\noindent\textbf{The role of AI in this proof.}
Odin Automatic AI Research Agent was used to discover the Riemannian metric and the proof that it has positive sectional curvature.

\vspace{3pt}

\noindent\textbf{Paper organization.}
Section~\ref{sec:preliminaries} introduces the
Cheeger reparametrization and the notation for zero-curvature planes.  In
Section~\ref{sec:construction} we state the quantitative construction and
the estimates used to prove it, and then deduce Theorem~\ref{thm:main}.  The
proofs are
given in Sections~\ref{sec:cheeger}--\ref{sec:nearby-planes}: first the
curvature estimate near the zero-curvature planes of the Cheeger-deformed
metric, then the connection metric, the construction of the complex-valued
symmetric $2$-tensor, and the perturbation argument.
Appendix~\ref{app:tensor-derivatives} gives the required estimates for
covariant derivatives, Appendix~\ref{app:tubular-minimization} proves the minimization
lemma, and Appendix~\ref{sec:topology} identifies the total space of the
circle bundle.

%% file: sections/preliminaries.tex
\section{Preliminaries}                                               \label{sec:preliminaries}

Equip each two-sphere with its unit round metric and standard orientation.
The area form of the $i$th factor is denoted by $\omega_i$, also after
pullback to the product, and is normalized by
$\int_{S^2}\omega_i=4\pi$.  Set $B=S^2\times S^2$, let $\bar g$ be the
product of the unit round metrics, set $\mathfrak g=\mathfrak{so}(3)$, and identify
$\mathfrak g$ with $\R^3$ by $z\mapsto(x\mapsto z\times x)$.
We identify $H^2(B;\Z)$ with $\Z^2$ using the integral classes
$[\omega_1/(4\pi)]$ and $[\omega_2/(4\pi)]$.
The bi-invariant inner product used in the
Cheeger deformation is the Euclidean product
$Q_{\mathfrak g}(z,w)=\langle z,w\rangle$.  Let $\bar g_t$ be the Cheeger
deformation of $\bar g$ at parameter $t$ for the diagonal
$\mathrm{SO}(3)$-action, using precisely the parameter convention in
\cite[Proposition~1.3]{Ziller2009}, and put
$g_t=\tfrac12\bar g_t$.  Thus the
final factor $1/2$ is a constant rescaling after the Cheeger deformation; it
does not change the parameter $t$.  The rescaling is chosen to match the
normalization of the curvature form in Section~\ref{sec:connection}.  By
Cheeger's construction, $g_t$ has
nonnegative sectional curvature~\cite{Cheeger1973}.

Unless otherwise specified, $c,C>0$ denote uniform constants whose values
may change from line to line; subscripted constants are fixed when introduced.

For $b\in S^2$, let
$T_b=(S^2\cap b^\perp)\times(S^2\cap b^\perp)$.  At
$(p_1,p_2)\in T_b$, set
$E_b=(b\times p_1,0)$ and $F_b=(0,b\times p_2)$, and denote the mixed
plane they span by $\Pi_b=\Span\{E_b,F_b\}$.
We write $\cZ\subset\Gr_2(TB)$ for the set of planes $\Pi_b$.  Each $T_b$
is a component of the fixed-point set of the
diagonal reflection across $b^\perp$, hence is totally geodesic.  It is
also flat.  Indeed, $E_b+F_b$ is tangent to a diagonal orbit and
$F_b-E_b$ is orthogonal to the diagonal $\mathrm{SO}(3)$-orbit.  The Cheeger
deformation multiplies the
squared length of the former by $(1+2t)^{-1}$ and leaves that of the latter
unchanged.  Thus the metric has constant coefficients with respect to angular
coordinates on the two equators.  This is the family of flat tori identified by
M\"uter~\cite[Satz~4.26]{Muter1987}.

Let $\cM\subset\Gr_2(TB)$ be the submanifold of mixed planes, namely the
planes spanned by vectors $E=(e,0)$ and $F=(0,f)$ tangent to different
sphere factors.  For $t=0$, it is exactly the zero set of the
sectional curvature function.

Let $\mathscr F$ be the compact bundle of ordered $g_0$-orthonormal
two-frames, and let
$\rho(X,Y)=\Span\{X,Y\}$ be its projection to $\Gr_2(TB)$.
We fix the auxiliary metric on $\Gr_2(TB)$ induced by $g_0$ and use its
distance function throughout.  The symbol $\dist_{\cM}$ denotes the intrinsic
distance for the induced metric on $\cM$.  Let $C_t$ be the positive definite,
$g_0$-self-adjoint endomorphism characterized by
$g_t(X,Y)=g_0(C_t X,Y)$, and let
$\mathcal R_t:\Gr_2(TB)\to\Gr_2(TB)$ be the Cheeger
reparametrization $\mathcal R_t(\Pi)=C_t^{-1}\Pi$.
We use the sign convention
\[
 R^g(X,Y)Z=\nabla_X^g\nabla_Y^g Z-\nabla_Y^g\nabla_X^g Z
 -\nabla_{[X,Y]}^g Z,
 \;\text{and}\;
 \Rm_g(X,Y,Z,W)=g(R^g(X,Y)Z,W),
\]
so that the unit sphere has positive sectional curvature.
For $\Pi=\rho(X,Y)$, set
\[
 K_t(\Pi)=
 \frac{\Rm_{g_t}(C_t^{-1}X,C_t^{-1}Y,C_t^{-1}Y,C_t^{-1}X)}
 {|C_t^{-1}X\wedge C_t^{-1}Y|_{g_t}^2}.
\]
The value is unchanged by an orthogonal change of the frame, and
$K_t(\Pi)=\sec_{g_t}(\mathcal R_t(\Pi))$.  Thus $K_t$ and all distances below are evaluated at
$\Pi$, before applying $\mathcal R_t$.  For small $t$,
$|C_t^{-1}X\wedge C_t^{-1}Y|_{g_t}^2$ is uniformly comparable to one on
$\mathscr F$.

We write $\sec_g(\Pi)$ for the sectional curvature of a plane $\Pi$ with
respect to a metric $g$.

%% file: sections/construction.tex
\section{The metric construction}                                             \label{sec:construction}

Let $\pi:P\to B$ be the principal $S^1$-bundle with first Chern class
$(1,1)$.  The following theorem states the construction and the estimates used
in its proof.  Its four assertions are proved in the sections that follow.

\begin{theorem}                                                                     \label{thm:construction}
There exists $t_0>0$ such that, for every $t\in(0,t_0)$, there exist a two-form
$\Omega_t$ on $B$, a connection one-form $\theta_t$ on $P$, a complex-valued
symmetric $2$-tensor $\mathcal Q_t$ on $P$, and a number
$\eps_0(t)>0$ with the following properties for every
$0<\eps<\eps_0(t)$.

\begin{enumerate}
\item The form $\Omega_t$ is symplectic and satisfies
$[\Omega_t/(2\pi)]=(1,1)\in H^2(B;\Z)$, and
$\dd\theta_t=\pi^*\Omega_t$.  The connection metric
$G_{t,\eps}=\pi^*g_t+\eps^2\theta_t\otimes\theta_t$ has nonnegative
sectional curvature.  If $\Pi_p^{\mathrm H}\subset\ker\theta_{t,p}$ denotes the
$\theta_t$-horizontal lift at $p\in P$ of a plane
$\Pi\subset T_{\pi(p)}B$, then the zero-curvature planes of $G_{t,\eps}$ are
exactly $\tZ=\{\Pi_p^{\mathrm H}:p\in P,\ \Pi\in\cZ\cap\Gr_2(T_{\pi(p)}B)\}$.

\item Let $V$ be the fundamental vertical vector field normalized by
$\theta_t(V)=1$.  There is an
$\eps$-independent auxiliary metric on $\Gr_2(TP)$ and an orthogonal
splitting of the normal bundle,
$\nu(\tZ)=\nu^{\mathrm h}(\tZ)\oplus\nu^{\mathrm v}(\tZ)$, where
$\nu^{\mathrm h}(\tZ)$ consists of variations through horizontal planes.  If
$\widetilde\Pi\in\tZ$ is contained in $T_pP$, then
$\nu^{\mathrm v}_{\widetilde\Pi}(\tZ)
=\operatorname{Hom}(\widetilde\Pi,\R V_p)$ consists
of variations obtained by taking graphs over $\widetilde\Pi$.  The norms
on $\nu^{\mathrm h}(\tZ)$ and $\nu^{\mathrm v}(\tZ)$ are induced,
respectively, by the fixed auxiliary metric on $\Gr_2(TB)$ and by $G_{t,1}$;
in particular, they do not depend on $\eps$.
The submanifold $\tZ$ lies in the critical set of the sectional curvature
function of $G_{t,\eps}$.  Let $H_\eps$ denote the restriction of its Hessian
along $\tZ$ to the normal bundle $\nu(\tZ)$.  There exists a constant
$\mu_t>0$, independent of $\eps$, such that, for every
$\widetilde\Pi\in\tZ$, $\xi\in\nu^{\mathrm h}_{\widetilde\Pi}(\tZ)$, and
$\vartheta\in\nu^{\mathrm v}_{\widetilde\Pi}(\tZ)$,
\[
 H_\eps((\xi,\vartheta),(\xi,\vartheta))
 \geq \mu_t\bigl(|\xi|^2+\eps^4|\vartheta|^2\bigr).
\]
Thus $H_\eps$ is positive definite, and $\tZ$ is a nondegenerate critical
submanifold of the sectional curvature function.

\item The tensor
$\mathcal Q_t\in\Gamma(\operatorname{Sym}^2T^*P\otimes\mathbb C)$ is
horizontal, meaning that $\mathcal Q_t(V,\cdot)=0$, and satisfies
$\Lie_V\mathcal Q_t=2i\mathcal Q_t$, where $\Lie_V$ denotes the Lie derivative
in the direction of $V$.  For every $\widetilde\Pi\in\tZ$, one can choose a
$G_{t,\eps}$-orthonormal basis $e_1,e_2$ of $\widetilde\Pi$, with dual
coframe $e^1,e^2$, such that, for some
$\delta_{\widetilde\Pi}\in\mathbb C$,
$\mathcal Q_t|_{\widetilde\Pi}
=\delta_{\widetilde\Pi}(e^1-ie^2)\otimes(e^1-ie^2)$.
The quantity $|\delta_{\widetilde\Pi}|$ is bounded below on $\tZ$ by a
positive constant depending only on $t$.
Moreover, if $\widetilde\Pi\subset T_pP$ lies over $\Pi_b$ and
$\Sigma_b$ is a local $\theta_t$-horizontal lift of $T_b$ through $p$ with
$T_p\Sigma_b=\widetilde\Pi$, then the restriction of $\mathcal Q_t$ to
$\Sigma_b$ is parallel for the induced flat metric.

\item For every $0<\eps<\eps_0(t)$, there is
$s_0(t,\eps)>0$ such that the metric
$G_{t,\eps}+s\operatorname{Re}\mathcal Q_t$ has strictly positive
sectional curvature whenever $0<|s|<s_0(t,\eps)$.
\end{enumerate}
\end{theorem}

Theorem~\ref{thm:construction} immediately implies the main theorem.
\begin{proof}[Proof of Theorem~\ref{thm:main}]
Choose $0<t<t_0$, $0<\eps<\eps_0(t)$, and
$0<|s|<s_0(t,\eps)$.  Part~(4) of
Theorem~\ref{thm:construction} gives a positively curved metric on $P$, and
Lemma~\ref{lem:topology} identifies $P$ diffeomorphically with
$S^2\times S^3$.
\end{proof}

The remainder of the paper is devoted to the proof of
Theorem~\ref{thm:construction}.  We indicate how its four assertions follow
from the key estimates.

\noindent\emph{Parts (1) and (2).}
Lemma~\ref{lem:cheeger-lower-bound}, obtained from M\"uter's formula for
reparametrized sectional curvature under the Cheeger deformation, gives
the estimate
$K_t(\Pi)\geq
c\bigl(\dist(\Pi,\cM)^2+t^3\dist(\Pi,\cZ)^2\bigr)$.
For the connection, set $\alpha_{(p_1,p_2)}(x_1,x_2)
=\langle p_1\times p_2,x_2-x_1\rangle$ and
$\Omega_t=\tfrac12(\omega_1+\omega_2+t\,\dd\alpha)$.
Lemma~\ref{lem:first-order-cancellation} proves that the derivative at $t=0$
of $\nabla^{g_t}\Omega_t$ vanishes on mixed planes.  Together with the
distance estimates in that lemma, this allows us to complete the square in
the curvature expression.
The argument in Section~\ref{sec:connection} then proves that the resulting
connection metric is nonnegatively curved and has $\tZ$ as its set of
zero-curvature planes.  Consequently, the sectional curvature function has
vanishing differential along $\tZ$.  Its quadratic expansion in the fixed
normal splitting is one half of $H_\eps$ and gives the lower bound in
part~(2), including the different powers of $\eps$ in the horizontal and
vertical normal directions.  Together,
Sections~\ref{sec:cheeger} and~\ref{sec:connection} prove parts~(1) and~(2).

\noindent\emph{Part (3).}
Section~\ref{sec:perturbation-tensor} constructs $\mathcal Q_t$ globally.
The resulting smooth symmetric $2$-tensor on $P$ vanishes whenever one
argument is vertical and satisfies $\Lie_V\mathcal Q_t=2i\mathcal Q_t$.
Proposition~\ref{prop:tensor-on-zero-curvature-planes} gives the restriction
of $\mathcal Q_t$ to each zero-curvature plane explicitly and proves the
uniform lower bound for its coefficient.  It also shows that its restriction
to the corresponding local horizontal lift of the flat torus is parallel.

\noindent\emph{Part (4).}
Set $h=\operatorname{Re}\mathcal Q_t$.  Section~\ref{sec:perturbation}
shows that the first variation of sectional curvature vanishes on $\tZ$,
whereas the Gauss equation gives a second-order coefficient
$a_\eps\geq\kappa\eps^{-2}-C_0$ for uniform constants $\kappa>0$ and
$C_0>0$.  In the tubular coordinates of
Section~\ref{sec:nearby-planes}, minimizing the Taylor expansion over
neighboring planes subtracts
$\tfrac12\langle H_\eps^{-1}r_\eps,r_\eps\rangle$ from the second-order coefficient.
Here $r_\eps$ is the differential of the first variation of sectional
curvature in directions normal to $\tZ$.
Lemmas~\ref{lem:first-variation-estimate} and~\ref{lem:block-inverse} bound this
term independently of $\eps$.  Hence the coefficient remaining after
minimization over nearby planes is positive once $\eps$ is sufficiently small.
Lemma~\ref{lem:tubular-minimization} then gives positivity on a fixed tubular
neighborhood of $\tZ$.  On the compact complement of a smaller tubular
neighborhood, positivity follows from continuity.
Thus part~(4) is proved in Sections~\ref{sec:perturbation}
and~\ref{sec:nearby-planes}; the proof is completed in the final paragraph of
Section~\ref{sec:nearby-planes}.

For some constant $c_0>0$, the choices can be made so that
$\eps_0(t)^2\leq c_0t^3$.  The
parameters are therefore fixed in the order
$0<t<t_0$, $0<\eps<\eps_0(t)$, and
$0<|s|<s_0(t,\eps)$.  In particular, the perturbation is chosen only after
the parameters $t$ and $\eps$ have been fixed.

%% file: sections/cheeger.tex
\section{Zero-curvature planes after the Cheeger deformation}                     \label{sec:cheeger}

We begin the proof of Theorem~\ref{thm:construction} with a lower bound for
sectional curvature on the base.  The description of the zero-curvature planes below is
the one used in
\cite[Proposition~2.3]{Bettiol2014} and
\cite[Appendix~B]{BrendleHung2026}.  Brendle and Hung prove a distance
estimate at a fixed Cheeger parameter in
\cite[Proposition~B.1 and Corollary~B.2]{BrendleHung2026}.  We record the
dependence on the small Cheeger parameter needed for the connection metric and
verify smoothness at $p_1=\pm p_2$.

\begin{lemma}                                                                       \label{lem:cheeger-lower-bound}
There exist $c>0$ and $t_0>0$ such that
\begin{equation}                                                                    \label{eq:cheeger-lower-bound}
 K_t(\Pi)\geq
 c\bigl(\dist(\Pi,\cM)^2+t^3\dist(\Pi,\cZ)^2\bigr),
\end{equation}
for every $\Pi\in\Gr_2(TB)$ and every $0<t<t_0$.  Moreover, for each
$t>0$, the zero-curvature planes of $g_t$ are exactly the planes in
$\cZ$, and $\cZ$ is a smooth compact embedded four-dimensional submanifold of the
Grassmann bundle; in particular, it is smooth at planes based at points where
$p_1=\pm p_2$.
\end{lemma}

\begin{proof}
For $X=(x_1,x_2)$ and $Y=(y_1,y_2)$, define
$\mathcal W:\mathscr F\to
\Lambda^2\R^3\oplus\Lambda^2\R^3$ by
$\mathcal W(X,Y)=(x_1\wedge y_1,x_2\wedge y_2)$.  The numerator in the
sectional curvature of the product metric is
\[
\Rm_{g_0}(X,Y,Y,X)=\tfrac12(|x_1\wedge y_1|^2+|x_2\wedge y_2|^2),
\]
with the wedge norms taken in the unit sphere metrics.  On the frame bundle,
the zero set of $\mathcal W$ is $\rho^{-1}(\cM)$.  After an orthogonal change
of frame, take $X=E=(e,0)$ and $Y=F=(0,f)$.  For a frame variation
$\dot X=(\dot x_1,\dot x_2)$ and
$\dot Y=(\dot y_1,\dot y_2)$, one has
$(\dd\mathcal W)_{(X,Y)}(\dot X,\dot Y)
=(e\wedge\dot y_1,\dot x_2\wedge f)$.
The linearized orthonormality constraints show that this vanishes precisely
on $T\rho^{-1}(\cM)$.  Thus $\dd\mathcal W$ is injective on the normal bundle
of $\rho^{-1}(\cM)$, and compactness gives
$K_0(\rho(X,Y))\geq c_1\dist(\rho(X,Y),\cM)^2$.
We next use M\"uter's formula for reparametrized sectional curvature under the
Cheeger deformation, as recorded in
\cite[Proposition~1.3]{Ziller2009}.  For $p=(p_1,p_2)$, let
$\mathfrak m_p$ be the $Q_{\mathfrak g}$-orthogonal complement of the
isotropy algebra.  The orbit tensor
$\mathcal P_p:\mathfrak m_p\to\mathfrak m_p$ of $\bar g$ is defined by
$Q_{\mathfrak g}(\mathcal P_p z,w)=\bar g(z^*,w^*)$, where
$z^*=(z\times p_1,z\times p_2)$.  Hence
$\mathcal P_p z=2z-\langle z,p_1\rangle p_1-\langle z,p_2\rangle p_2$.
For $X\in T_pB$, let $X_{\mathfrak m}\in\mathfrak m_p$ be the unique element
such that $X-X_{\mathfrak m}^*$ is orthogonal to the orbit through $p$.  The
identity
$\mathcal P_p X_{\mathfrak m}=p_1\times x_1+p_2\times x_2$ also holds on
singular orbits.  Define
$\Phi(X,Y)=[\mathcal P_p X_{\mathfrak m},
\mathcal P_p Y_{\mathfrak m}]$ and
$\phi(\rho(X,Y))=|\Phi(X,Y)|_{Q_{\mathfrak g}}$.  The function
$\phi$ is well defined
because an orthogonal change of frame multiplies $\Phi$ by its
determinant.  Write $\Pi=\rho(X,Y)$.  Our parameter $t$ is
exactly the parameter in that proposition; only after the
deformation do we rescale by $1/2$.  Therefore
\begin{equation}                                                                    \label{eq:cheeger-formula}
 \Rm_{g_t}(C_t^{-1}X,C_t^{-1}Y,C_t^{-1}Y,C_t^{-1}X)=\Rm_{g_0}(X,Y,Y,X)+\frac{t^3}{8}\phi(\Pi)^2+z_t(X,Y),
\end{equation}
where $z_t\geq0$ denotes one half of the nonnegative remainder in M\"uter's
formula for $\bar g_t$.  The coefficient is $1/8$,
rather than $1/4$, because
the final rescaling multiplies the $(0,4)$-curvature identity by $1/2$;
$C_t$ is unchanged.

Within $\cM$, the zero set of $\phi$ is exactly $\cZ$.  To verify
smoothness, let
$V_2(\R^3)=\{(p,e)\in S^2\times S^2:\langle p,e\rangle=0\}$ be the Stiefel
manifold of ordered orthonormal two-frames.  Give it the metric induced from
$S^2\times S^2$, and give
$\widehat{\cM}=V_2(\R^3)\times V_2(\R^3)$ the product metric; denote its
distance function by $\dist_{\widehat{\cM}}$.  The map
$\pi_{\cM}:\widehat{\cM}\to\cM$ defined by
$\pi_{\cM}(p_1,\widehat e,p_2,\widehat f)
=\Span\{(\widehat e,0),(0,\widehat f)\}$ is a fourfold covering.  Put
$u_1=p_1\times\widehat e$ and
$u_2=p_2\times\widehat f$.  The map
$\tau:\widehat{\cM}\to S^2\times S^2$ defined by $\tau=(u_1,u_2)$ is a submersion,
because $(p,e)\mapsto p\times e$ is a submersion from $V_2(\R^3)$ to $S^2$.
For the $g_0$-orthonormal frame adapted to the product factors,
$E=(\sqrt2\,\widehat e,0)$ and $F=(0,\sqrt2\,\widehat f)$, one has
$\Phi(E,F)=2u_1\times u_2$.  Consequently,
\[
 \pi_{\cM}^{-1}(\cZ)=
 \tau^{-1}\bigl(\{(u,u):u\in S^2\}\cup\{(u,-u):u\in S^2\}\bigr).
\]
The diagonal and antidiagonal in $S^2\times S^2$ are disjoint embedded
submanifolds of codimension two.  Their inverse image under $\tau$ is
therefore a smooth codimension-two submanifold of $\widehat{\cM}$.  It is
invariant under the deck transformations of $\pi_{\cM}$, so its quotient is a
smooth embedded four-dimensional submanifold $\cZ\subset\cM$.  This argument
also applies when $p_1=\pm p_2$.

Transversality and compactness give
$|u_1\times u_2|\geq c_2\dist_{\widehat{\cM}}
((p_1,\widehat e,p_2,\widehat f),\pi_{\cM}^{-1}(\cZ))$.
Since $\pi_{\cM}$ is a finite covering and
$\phi(\pi_{\cM}(p_1,\widehat e,p_2,\widehat f))=2|u_1\times u_2|$, it follows that,
after decreasing $c_2$,
$\phi\geq c_2\dist_{\cM}(\,\cdot\,,\cZ)$ on $\cM$.
A tubular projection extends the comparison to a neighborhood of $\cM$.
Smoothness of the map $\Phi$ on the frame bundle, together with the invariance of
its norm under orthogonal changes of frame and the preceding estimate, gives
$\dist(\Pi,\cZ)\leq
C\bigl(\dist(\Pi,\cM)+\phi(\Pi)\bigr)$.
Combining this inequality with the estimate for $K_0$ above and
\eqref{eq:cheeger-formula}, and using $t^3\leq1$, proves
\eqref{eq:cheeger-lower-bound} in a fixed tubular neighborhood of $\cM$.
Compactness gives the same estimate on the complement.  If $K_t(\Pi)=0$, then the left-hand
side of \eqref{eq:cheeger-formula} vanishes, so the nonnegative terms
$\Rm_{g_0}(X,Y,Y,X)$ and $\phi(\Pi)^2$ vanish separately.
Thus $\Pi\in\cM$ and then $\Pi\in\cZ$.  Conversely, on every
$\Pi_b$ the vectors $E_b+F_b$ and $F_b-E_b$ are eigenvectors of $C_t$,
by the decomposition into the orbit direction and its orthogonal complement
used above.  Hence
$C_t^{-1}\Pi_b=\Pi_b$.  This plane is tangent to the flat, totally geodesic
torus $T_b$ for $g_t$, so $K_t(\Pi_b)=0$.  In
particular, all three nonnegative terms in \eqref{eq:cheeger-formula},
including $z_t$, vanish on $\cZ$.  Thus $\cZ$ is the zero set of $K_t$ for
$t>0$.  Since
reparametrization is a diffeomorphism of the Grassmann bundle and preserves
$\cZ$, the zero-curvature planes of $g_t$ are exactly $\cZ$ as well.
\end{proof}

%% file: sections/connection.tex
\section{The connection metric and its curvature}                                  \label{sec:connection}

We now choose a symplectic representative of the primitive class $(1,1)$
that vanishes on every flat torus $T_b$.  For a principal circle bundle,
nondegeneracy of the curvature form is precisely Weinstein's fatness
condition~\cite{Weinstein1980}.  The additional vanishing and derivative
estimates below allow us to verify the curvature inequalities for a
connection metric even when the base has zero-curvature planes.  Related
curvature estimates for connection metrics on vector and sphere bundles appear in
\cite{StrakeWalschap1990,Tapp2003}.  Define the
one-form $\alpha$ on $B$ by
$\alpha_{(p_1,p_2)}(x_1,x_2)=\langle p_1\times p_2,x_2-x_1\rangle$, and set
$\Omega_t=\tfrac12(\omega_1+\omega_2+t\,\dd\alpha)$.
This form is closed and satisfies
$\left[\Omega_t/(2\pi)\right]=(1,1)$ in $H^2(B;\Z)$.
Because $\Omega_0$ is symplectic, $\Omega_t$ is symplectic for all
sufficiently small $t$.  Compactness and smooth dependence on $t$ give
constants $0<q_*\leq q^*<\infty$, uniform for small $t$, such that
$q_*\leq q_t(X):=|\iota_X\Omega_t|_{g_t}^2\leq q^*$ for every $g_t$-unit
vector $X$.

For $(X_0,Y_0)\in\mathscr F$, put $\Pi=\rho(X_0,Y_0)$ and smoothly
orthonormalize $C_t^{-1}X_0,C_t^{-1}Y_0$ to obtain an ordered
$g_t$-orthonormal pair $X,Y$ spanning $\mathcal R_t(\Pi)$.  Define
$\mathfrak a_t(X,Y)=(\nabla^{g_t}_X\Omega_t)(X,Y)$ and
$\mathfrak b_t(X,Y)=\Omega_t(X,Y)$.  We suppress the frame arguments from
$\mathfrak a_t$, $\mathfrak b_t$, and $q_t(X)$ below.  These quantities
depend on the chosen frame, whereas $K_t$, $\dist(\Pi,\cM)$, and
$\dist(\Pi,\cZ)$ are evaluated at $\Pi$, before applying $\mathcal R_t$.
All estimates below are uniform in the chosen smooth orthonormalization.

\begin{lemma}                                                                       \label{lem:first-order-cancellation}
On $\mathscr F$, one has $\mathfrak a_0=0$, and
$\left.\partial_t\mathfrak a_t\right|_{t=0}=0$ on
$\rho^{-1}(\cM)$.  Moreover, $\mathfrak a_t=\mathfrak b_t=0$ on
$\rho^{-1}(\cZ)$.  Consequently,
\begin{equation}                                                                    \label{eq:a-b-distance-estimates}
 |\mathfrak a_t|\leq C\bigl(t\dist(\Pi,\cM)+t^2\dist(\Pi,\cZ)\bigr)
 \;\text{and}\; |\mathfrak b_t|\leq C\dist(\Pi,\cZ).
\end{equation}
\end{lemma}

\begin{proof}
Let $\bar g$ be the unit product metric, so that
$g_0=\frac12\bar g$.  Set $J_i x=p_i\times x$,
$J=J_1\oplus J_2$, and $\beta(x_1,x_2)=J_1x_1+J_2x_2$.
Throughout this first-order calculation, $\nabla=\nabla^{g_0}$, and an
overdot denotes $\left.\partial_t\right|_{t=0}$.
Then $\Omega_0(X,Y)=g_0(JX,Y)$.  Differentiating
$g_t(X,Y)=g_0(C_tX,Y)$ and the definition of $\Omega_t$ at $t=0$ gives
$\dot g(X,Y)=-\tfrac12\langle\beta(X),\beta(Y)\rangle$ and
$\dot\Omega=\tfrac12\dd\alpha$.
Terms arising from the variation of the frame do not contribute because they
are contracted with $\nabla^{g_0}\Omega_0=0$.

Choose extensions that are parallel at the point with respect to the product
connection.  Then
$(\nabla_Z\beta)(X)=z_1\times x_1+z_2\times x_2$, and hence
\begin{equation}\label{eq:metric-variation-derivative}
 (\nabla_Z\dot g)(X,Y)=-\frac12\{\langle(\nabla_Z\beta)(X),\beta(Y)\rangle+\langle\beta(X),(\nabla_Z\beta)(Y)\rangle\}.
\end{equation}
For $E=(e,0)$ and $F=(0,f)$, direct differentiation of $\alpha$ gives
$\dd\alpha(E,F)
=\langle e\times p_2,f\rangle+\langle p_1\times f,e\rangle$.
Writing $D$ for the flat connection on $\R^3$, use
$D_e e=-|e|^2p_1$ and $D_f f=-|f|^2p_2$ to obtain
\begin{equation}                                                                    \label{eq:omega-variations}
 (\nabla_E\dot\Omega)(E,F)=-\frac12|e|^2\langle p_1\times p_2,f\rangle
 \;\text{and}\; (\nabla_F\dot\Omega)(F,E)=\frac12|f|^2\langle p_1\times p_2,e\rangle.
\end{equation}
Let $\dot{\nabla}=\left.\partial_t\nabla^{g_t}\right|_{t=0}$.  The
Levi-Civita variation
formula is
\begin{equation}                                                                    \label{eq:connection-variation}
 2g_0(\dot{\nabla}_X Y,Z)
 =(\nabla_X\dot g)(Y,Z)+(\nabla_Y\dot g)(X,Z)
 -(\nabla_Z\dot g)(X,Y).
\end{equation}
Substitution of \eqref{eq:metric-variation-derivative} into
\eqref{eq:connection-variation}, first with $Z=JF$ and then with
$Z=JE$, yields
\begin{equation}                                                                    \label{eq:connection-variations}
 \begin{aligned}
 \Omega_0(\dot{\nabla}_E E,F)&=0
 &\text{and } \Omega_0(E,\dot{\nabla}_E F)&=-\frac12|e|^2\langle p_1\times p_2,f\rangle,\\
 \Omega_0(\dot{\nabla}_F F,E)&=0
 &\text{and } \Omega_0(F,\dot{\nabla}_F E)&=\frac12|f|^2\langle p_1\times p_2,e\rangle.
 \end{aligned}
\end{equation}
For instance,
$2g_0(\dot{\nabla}_E F,J E)=-|e|^2\langle p_1,J_2f\rangle
=-|e|^2\langle p_1\times p_2,f\rangle$, which is the second identity in
\eqref{eq:connection-variations}.

Let $\mathcal V_\Omega$ denote the first variation of
$\nabla^{g_t}\Omega_t$ at $t=0$; explicitly,
\[
 \mathcal V_\Omega(Z;X,Y)=(\nabla_Z\dot\Omega)(X,Y)
 -\Omega_0(\dot{\nabla}_Z X,Y)-\Omega_0(X,\dot{\nabla}_Z Y).
\]
Equations \eqref{eq:omega-variations} and \eqref{eq:connection-variations} give
$\mathcal V_\Omega(E;E,F)=\mathcal V_\Omega(F;F,E)=0$.
The tensor $\mathcal V_\Omega$ is linear in its first slot and alternating in the last two.
If $X=aE+bF$ and $Y=-bE+aF$ form an oriented orthonormal rotation
of the frame $E,F$, then $X\wedge Y=E\wedge F$, and therefore
$\mathcal V_\Omega(X;X,Y)=\mathcal V_\Omega(aE+bF;E,F)
=a\mathcal V_\Omega(E;E,F)-b\mathcal V_\Omega(F;F,E)=0$.
The calculation also applies when $p_1=\pm p_2$.  This proves the first two
assertions of the lemma.

On $T_b$, both $\omega_i$ vanish because each factor contributes only
one tangent direction.  Moreover, $p_1\times p_2$ is parallel to $b$,
while every tangent vector to $T_b$ is perpendicular to $b$.  Hence
$\alpha|_{T_b}=0$ and $\Omega_t|_{T_b}=0$.  Since $T_b$ is totally geodesic,
this implies $\mathfrak a_t=\mathfrak b_t=0$ on $\rho^{-1}(\cZ)$.

Finally, write
$\mathfrak a_t=t\mathfrak a^{(1)}+t^2\mathfrak a^{(2)}_t$.
The vanishing of $\mathfrak a^{(1)}$ on $\rho^{-1}(\cM)$ and smoothness in a
tubular neighborhood give
$|\mathfrak a^{(1)}|\leq C\dist(\rho(\,\cdot\,),\cM)$.  Both
$\mathfrak a_t$ and $\mathfrak a^{(1)}$ vanish on
$\rho^{-1}(\cZ)$, so
$|\mathfrak a^{(2)}_t|\leq C\dist(\rho(\,\cdot\,),\cZ)$.  This proves the
first estimate in \eqref{eq:a-b-distance-estimates}; the second follows from
the vanishing of $\mathfrak b_t$ on the smooth compact submanifold
$\rho^{-1}(\cZ)$.
\end{proof}

We now derive the strict inequality needed below.
Lemma~\ref{lem:cheeger-lower-bound}, Lemma~\ref{lem:first-order-cancellation},
and the bounds for $q_t$ above give
$K_t(\Pi)q_t\geq cq_*\bigl(\dist(\Pi,\cM)^2
+t^3\dist(\Pi,\cZ)^2\bigr)$ and
$\mathfrak a_t^2\leq C\bigl(t^2\dist(\Pi,\cM)^2
+t^4\dist(\Pi,\cZ)^2\bigr)$.  After decreasing $t_0$ if necessary, assume
$Ct_0^2\leq cq_*/2$ and $Ct_0\leq cq_*/2$.  These choices make each term in
the upper bound for $\mathfrak a_t^2$ at most one half of the corresponding
term in the lower bound for $K_t(\Pi)q_t$.  After renaming the constant, we obtain
\begin{equation}                                                                    \label{eq:base-strict-inequality}
 K_t(\Pi)q_t-\mathfrak a_t^2
 \geq c\bigl(\dist(\Pi,\cM)^2+t^3\dist(\Pi,\cZ)^2\bigr).
\end{equation}
\subsection{The nonnegative connection metric}

In the quotient model of Section~\ref{sec:perturbation-tensor}, let
$\theta_0$ be the connection
induced by the sum of the two Hopf connection forms, and set
$\theta_t=\theta_0+\tfrac t2\pi^*\alpha$.  Then
$\dd\theta_t=\pi^*\Omega_t$.  For $\eps>0$, define
$G_{t,\eps}=\pi^*g_t+\eps^2\theta_t\otimes\theta_t$.
On horizontal vectors, $g_t$ and $\Omega_t$ denote their pullbacks to $P$.
Let $V$ be the fundamental vertical vector field normalized by
$\theta_t(V)=1$, and put $U=\eps^{-1}V$, so that $U$ is unit.
In this subsection, $\nabla=\nabla^{G_{t,\eps}}$.
For all comparisons as $\eps\to0$, however, we use the fixed metric
$G_{t,1}=\pi^*g_t+\theta_t\otimes\theta_t$ on $P$.  Thus
$V$, rather than $U$, is unit with respect to the auxiliary metric.  In the normal Hessian
argument below, we use $G_{t,1}$ to construct an auxiliary metric compatible
with the normal splitting and tubular coordinates on $\Gr_2(TP)$.  All
resulting normal bundles, distances, and covector norms are independent of
$\eps$.

Every tangent two-plane in $P$ has an orthonormal basis
$X,\lambda_1Y+\lambda_2U$, where $X,Y$ are horizontal and
$g_t$-orthonormal and $\lambda_1^2+\lambda_2^2=1$.  If $\lambda_1=0$,
choose any horizontal unit vector $Y$ orthogonal to $X$; otherwise, $Y$ is
determined up to sign by the plane.  In either case, let $\Pi$ satisfy
$\Span\{X,Y\}=\mathcal R_t(\Pi)$.  The curvature formula below is independent
of the choice of $Y$ when $\lambda_1=0$.  The Koszul formula gives
\begin{equation}                                                                    \label{eq:connection-koszul}
 \begin{aligned}
 \nabla_X Y&=(\nabla_X^{g_t}Y)^{\mathrm H}-\frac{\eps}{2}\Omega_t(X,Y)U,
 &\nabla_X U=\nabla_U X&=\frac{\eps}{2}(\iota_X\Omega_t)^\sharp,
 &\nabla_U U&=0.
 \end{aligned}
\end{equation}
Here $\sharp$ is taken with respect to $g_t$.
These identities also follow by specializing O'Neill's submersion equations;
see~\cite[pp.~465--466]{ONeill1966}.  Suppressing the metric from the notation
for the curvature tensor in the next display, they give
\[
 \Rm(X,Y,Y,X)=K_t(\Pi)-\frac{3\eps^2}{4}\mathfrak b_t^2,\;
 \Rm(X,U,U,X)=\frac{\eps^2}{4}q_t,\;\text{and}\;
 \Rm(X,Y,U,X)=-\frac{\eps}{2}\mathfrak a_t.
\]
Here $\mathfrak a_t$, $\mathfrak b_t$, and $q_t$ are the quantities associated
with $X,Y$ above, while $K_t$ and all Grassmann distances are evaluated at
$\Pi$.

The preceding identities express sectional curvature as a quadratic form in
$\lambda_1$ and $\lambda_2$.  For a principal circle bundle with constant
fiber length, its nonnegativity condition is the specialization of the
criterion of Chaves, Derdzinski, and Rigas
\cite[main theorem, condition~(ii)]{ChavesDerdzinskiRigas1992}.  The exact
necessary-and-sufficient inequalities for invariant metrics on circle bundles are
given in~\cite[Lemma~3.2]{ShankarTappTuschmann2005}.  We complete the square
explicitly because the equality case, rather than the criterion itself, is
essential here.

The sectional curvature of the plane spanned by
$X,\lambda_1Y+\lambda_2U$ is exactly
\begin{align}
 \sec_{G_{t,\eps}}\bigl(\Span\{X,\lambda_1Y+\lambda_2U\}\bigr)
 &=\left(K_t(\Pi)-\frac{3\eps^2}{4}\mathfrak b_t^2\right)\lambda_1^2
   -\eps\mathfrak a_t\lambda_1\lambda_2+\frac{\eps^2q_t}{4}\lambda_2^2 \notag\\
 &=\frac{\eps^2q_t}{4}\left(\lambda_2-\frac{2\lambda_1\mathfrak a_t}{\eps q_t}\right)^2
   +\lambda_1^2\left(\frac{K_t(\Pi)q_t-\mathfrak a_t^2}{q_t}
   -\frac{3\eps^2}{4}\mathfrak b_t^2\right).                       \label{eq:connection-curvature-square}
\end{align}
By \eqref{eq:a-b-distance-estimates} and
\eqref{eq:base-strict-inequality}, there is $c_0>0$ such that
$0<\eps^2\leq c_0t^3$ implies
\begin{equation}                                                                    \label{eq:connection-lower-bound}
 \frac{K_t(\Pi)q_t-\mathfrak a_t^2}{q_t}
 -\frac{3\eps^2}{4}\mathfrak b_t^2
 \geq c\bigl(\dist(\Pi,\cM)^2+t^3\dist(\Pi,\cZ)^2\bigr).
\end{equation}
Indeed, division of \eqref{eq:base-strict-inequality} by $q_t\leq q^*$ leaves
a fixed positive multiple of
$\dist(\Pi,\cM)^2+t^3\dist(\Pi,\cZ)^2$, while
$\mathfrak b_t^2\leq C\dist(\Pi,\cZ)^2$.  Choosing $c_0$ so that the latter
term, bounded by $C c_0t^3\dist(\Pi,\cZ)^2$, is at most half of the former proves
\eqref{eq:connection-lower-bound}.
It follows from \eqref{eq:connection-curvature-square} that
$G_{t,\eps}$ has nonnegative sectional curvature.  Equality occurs exactly
when both the square and the second term in
\eqref{eq:connection-curvature-square} vanish.  The square cannot vanish when
$\lambda_1=0$, because then $|\lambda_2|=1$.  Hence $\lambda_1\neq0$, and
\eqref{eq:connection-lower-bound} forces $\Pi\in\cZ$.  Then
$\mathfrak a_t=0$, so the square forces $\lambda_2=0$.  Conversely these two
conditions make every term vanish.  Thus the set of zero-curvature planes of
$G_{t,\eps}$ is the smooth compact embedded submanifold
$\tZ\subset\Gr_2(TP)$ consisting of the $\theta_t$-horizontal lifts of the
planes in $\cZ$.  Since the sectional curvature function is nonnegative and
vanishes on $\tZ$, its differential vanishes there.

\subsection{The normal Hessian}

The same completion of squares determines the dependence of the normal
Hessian along $\tZ$ on $\eps$, but we first choose normal coordinates
independent of $\eps$.
Fix a tubular neighborhood parametrization for
$\cZ\subset\Gr_2(TB)$.  Along each radial curve, lift the curve of
footpoints $\theta_t$-horizontally to $P$, apply $\mathcal R_t$ to the
corresponding base planes, and lift those planes into $\ker\theta_t$.  This
gives an $S^1$-equivariant tubular parametrization whose image consists of
horizontal planes.  Taking graphs of maps from these planes to $\R V$ extends
it to a tubular parametrization of $\tZ$ in $\Gr_2(TP)$.  Along the zero
section, denote the
rank-four subbundle of variations through horizontal planes by
$\nu^{\mathrm h}(\tZ)$, and denote the rank-two subbundle consisting of
variations represented by graphs into $\R V$ by $\nu^{\mathrm v}(\tZ)$.  At
$\widetilde\Pi\subset T_pP$, the latter has fiber
$\nu^{\mathrm v}_{\widetilde\Pi}(\tZ)
=\operatorname{Hom}(\widetilde\Pi,\R V_p)$.  These subbundles form a fixed
complement to $T\tZ$.  Choose an auxiliary metric on $\Gr_2(TP)$ that makes
$T\tZ\oplus\nu^{\mathrm h}(\tZ)\oplus\nu^{\mathrm v}(\tZ)$ orthogonal and
gives the last two summands the norms induced by the fixed base metric and
$G_{t,1}$, respectively.  Thus
$\nu(\tZ)=\nu^{\mathrm h}(\tZ)\oplus\nu^{\mathrm v}(\tZ)$, with all summands and
norms independent of $\eps$.

Fix $z_0\in\tZ$.  Choose Fermi coordinates in the horizontal normal
directions and write $\xi\in\nu^{\mathrm h}_{z_0}(\tZ)$.  Let $\Pi(\xi)$ be
the corresponding base plane before Cheeger reparametrization, and choose a
smooth oriented $g_t$-orthonormal frame $X(\xi),Y(\xi)$ spanning
$\mathcal R_t(\Pi(\xi))$.  Use the same letters for their horizontal lifts
and parametrize variations into $\R V$ by
$\widetilde\Pi(\xi,\vartheta)=
\Span\{X(\xi)+\vartheta_1V,Y(\xi)+\vartheta_2V\}$.
Set $w=\eps^2\vartheta$.  As observed above, the differential of the
sectional curvature function vanishes on $\tZ$.  Its quadratic Taylor term
therefore represents one half of the restriction of the Hessian to
$\nu(\tZ)$, independently of the chosen extension of these coordinates.
The quadratic Taylor term of the
sectional curvature at $z_0$, expressed in the rescaled variables
$(\xi,w)$, has the following form.  There exist a linear map
$L_{z_0}:\nu^{\mathrm h}_{z_0}(\tZ)\to\R^2$, a symmetric bilinear form
$q_{z_0}$ on $\R^2$, and a symmetric bilinear form $S_{\eps,z_0}$ on
$\nu^{\mathrm h}_{z_0}(\tZ)$ such that, for all $\xi$ and $w$,
\begin{equation}                                                                    \label{eq:normal-hessian-square}
 \begin{aligned}
 Q_{\eps,z_0}(\xi,w)&=\frac14q_{z_0}(w-L_{z_0}\xi,w-L_{z_0}\xi)+S_{\eps,z_0}(\xi,\xi),\\
 &\text{where }q_{z_0}(w,w)\geq q_*|w|^2\text{ and }
 S_{\eps,z_0}(\xi,\xi)\geq ct^3|\xi|^2.
 \end{aligned}
\end{equation}
The lower bound for $S_{\eps,z_0}$ is proved below and is uniform in $z_0$
and $\eps$ for $0<\eps<\eps_0(t)$.  The vertical quadratic form is
$q_{z_0}(w,w)=
|w_1\iota_{Y(0)}\Omega_t-w_2\iota_{X(0)}\Omega_t|_{g_t}^2$.
It is positive definite: if the displayed one-form
vanishes, the nondegeneracy of $\Omega_t$ gives
$w_1Y(0)-w_2X(0)=0$, hence $w=0$.  In fact, the lower bound $q_t\geq q_*$,
applied to $w_1Y(0)-w_2X(0)$, gives the stated lower bound for $q_{z_0}$.

We derive the other terms, including their normalization.  For $w\neq0$, set
$v=w/|w|=(v_1,v_2)\in S^1$ and
$X_v(\xi)=v_2X(\xi)-v_1Y(\xi)$ and
$Y_v(\xi)=v_1X(\xi)+v_2Y(\xi)$.  Then
$X_v(\xi)\wedge Y_v(\xi)=X(\xi)\wedge Y(\xi)$, and
\[
 \widetilde\Pi(\xi,\eps^{-2}w)
 =\Span\{X_v(\xi),Y_v(\xi)+\eps^{-1}|w|U\}.
\]
Set
$\mathcal K_\eps(\xi)=K_t(\Pi(\xi))-\tfrac{3\eps^2}{4}
\Omega_t(X(\xi),Y(\xi))^2$,
$\mathfrak a_v(\xi)=(\nabla_{X_v(\xi)}^{g_t}\Omega_t)(X_v(\xi),Y_v(\xi))$, and
$q_v(\xi)=|\iota_{X_v(\xi)}\Omega_t|_{g_t}^2$.
The curvature formula gives the exact identity
\[
 \sec_{G_{t,\eps}}\bigl(\widetilde\Pi(\xi,\eps^{-2}w)\bigr)
 =\frac{\mathcal K_\eps(\xi)-|w|\mathfrak a_v(\xi)+(|w|^2/4)q_v(\xi)}
 {1+\eps^{-2}|w|^2}.
\]
At $\xi=0$, one has $\mathcal K_\eps(0)=0$ and
$\mathfrak a_v(0)=0$ for every $v\in S^1$.  Put
$H^{\mathrm h}_{\eps,z_0}
=\tfrac12(\operatorname{Hess}\mathcal K_\eps)_0$ and
$\ell_v=(\dd\mathfrak a_v)_0$.  Because
$X_v(\xi)\wedge Y_v(\xi)=X(\xi)\wedge Y(\xi)$, the identity
$\mathfrak a_v(\xi)=(\nabla_{X_v(\xi)}^{g_t}\Omega_t)(X(\xi),Y(\xi))$ shows that the values of
$\mathfrak a_v$, and hence of $\ell_v$, on $S^1$ extend linearly in $v\in\R^2$.  Therefore
$|w|\ell_{w/|w|}(\xi)$ extends across $w=0$ to a bilinear form
$\Lambda_{z_0}(\xi,w)$.  Its quadratic Taylor term is
\[
 Q_{\eps,z_0}(\xi,w)
 =H^{\mathrm h}_{\eps,z_0}(\xi,\xi)-\Lambda_{z_0}(\xi,w)
   +\frac14q_{z_0}(w,w).
\]
Since $q_{z_0}$ is positive definite, there is a unique linear map $L_{z_0}$
such that $q_{z_0}(L_{z_0}\xi,w)=2\Lambda_{z_0}(\xi,w)$.  Set
$S_{\eps,z_0}=H^{\mathrm h}_{\eps,z_0}
-\tfrac14q_{z_0}(L_{z_0}\,\cdot,L_{z_0}\,\cdot)$.  Completing the square now gives
\eqref{eq:normal-hessian-square}.
It remains to prove the asserted lower bound for $S_{\eps,z_0}$.  Apply
\eqref{eq:connection-lower-bound} to the curve
$\tau\mapsto\Pi(\tau\xi)$ in the chosen normal coordinates on
$\Gr_2(TB)$.  For each fixed $v\in S^1$, use the frame
$X_v(\tau\xi),Y_v(\tau\xi)$.  Since $\xi$ ranges over the chosen complement
of $T\cZ$, compactness and uniform equivalence of the auxiliary metrics give
$\dist(\Pi(\tau\xi),\cZ)^2\geq c\tau^2|\xi|^2+o(\tau^2)$.
After division by $\tau^2$ and passage to the limit, estimate
\eqref{eq:connection-lower-bound}, uniformly for $v\in S^1$, gives
\[
 H^{\mathrm h}_{\eps,z_0}(\xi,\xi)
 -\frac{\ell_v(\xi)^2}{q_v(0)}
 \geq ct^3|\xi|^2.
\]
The Rayleigh quotient formula is
$\tfrac14q_{z_0}(L_{z_0}\xi,L_{z_0}\xi)
=\sup_{v\in S^1}\ell_v(\xi)^2/q_v(0)$.
Taking the supremum in this inequality proves
$S_{\eps,z_0}(\xi,\xi)\geq ct^3|\xi|^2$.  Compactness makes this uniform on $\tZ$.  The
coefficients of the rescaled quadratic polynomial are uniformly bounded for
fixed small $t$, and the lower bound for $q_{z_0}$ is uniform.  Hence the linear change
of variables $(\xi,w)\mapsto(\xi,w-L_{z_0}\xi)$ in
\eqref{eq:normal-hessian-square} and its inverse have operator norm bounded
independently of $z_0$ and $\eps$.  These bounds apply to the rescaled
variable $w=\eps^2\vartheta$; no bound for the same change expressed in
$\vartheta$ is needed.
Consequently, there is $\mu_t>0$, independent of $z_0$ and $\eps$ for
$0<\eps<\eps_0(t)$, such that
$Q_{\eps,z_0}(\xi,w)\geq \mu_t(|\xi|^2+|w|^2)$.
Let $H_\eps$ denote the restriction along $\tZ$ of the Hessian of the
sectional curvature function, expressed in the fixed normal splitting.  Then
$H_\eps((\xi,\vartheta),(\xi,\vartheta))
=2Q_{\eps,z_0}(\xi,\eps^2\vartheta)$.  Therefore
\begin{equation}                                                                    \label{eq:normal-hessian-lower-bound}
 H_\eps((\xi,\vartheta),(\xi,\vartheta))
 \geq \mu_t\bigl(|\xi|^2+\eps^4|\vartheta|^2\bigr).
\end{equation}
All norms in \eqref{eq:normal-hessian-lower-bound} are the fixed auxiliary norms.
The right-hand side is positive for every nonzero normal vector.  Thus
$H_\eps$ is positive definite, and $\tZ$ is a nondegenerate critical
submanifold of the sectional curvature function.
This proves parts~(1) and~(2) of Theorem~\ref{thm:construction}.

%% file: sections/perturbation-tensor.tex
\section{Construction of the complex-valued symmetric tensor}           \label{sec:perturbation-tensor}

We now construct the tensor used to perturb the connection metric.  We use the
quotient model $P=(S^3\times S^3)/S^1$, where $e^{i\lambda}$ acts by
$e^{i\lambda}\cdot(u,v)=(e^{i\lambda}u,e^{-i\lambda}v)$.
The projection to $B$ is induced by the two Hopf maps.  We use the
convention $\pi_{\mathrm H}(u_1,u_2)=
(2\operatorname{Re}(\bar u_1u_2),2\operatorname{Im}(\bar u_1u_2),
|u_1|^2-|u_2|^2)$.
With the standard orientation on $S^2$, and with the Hermitian product
linear in its second argument, the Hopf connection
$\eta=-i\langle u,\dd u\rangle_{\mathbb C^2}$ satisfies
$\eta(iu)=1$ and $\dd\eta=\tfrac12\pi_{\mathrm H}^*\omega$, where
$\omega$ is the area form of the oriented unit sphere.  These conventions
fix both the sign of the Chern class and the sign in the connection used below.

For $u\in S^3\subset\mathbb C^2$, define the complex-valued one-form
$\zeta_u(\dot u)=u_1\dot u_2-u_2\dot u_1
=\det_{\mathbb C}(u,\dot u)$.
Its kernel on $T_u S^3$ is exactly the vertical line $\R\,iu$ of the Hopf
fibration, and it satisfies the transformation law
$\zeta_{e^{i\lambda}u}(e^{i\lambda}\dot u)
=e^{2i\lambda}\zeta_u(\dot u)$.
Moreover, $\zeta_{Au}(A\dot u)=\zeta_u(\dot u)$ for every
$A\in\mathrm{SU}(2)$.
Thus $\zeta$ is $\mathrm{SU}(2)$-invariant, while scalar multiplication by
$e^{i\lambda}$ multiplies it by $e^{2i\lambda}$.  Let $\eta_j$ and $\zeta_j$
denote the pullbacks of $\eta$ and $\zeta$, respectively, from the $j$th
factor of $S^3\times S^3$.
We normalize the symmetrized tensor product by
$\beta\odot\gamma=\tfrac12(\beta\otimes\gamma+\gamma\otimes\beta)$ and set
$\Psi=\zeta_1\odot\zeta_2$.
The transformation law for $\zeta$ shows that $\Psi$ is invariant under the
circle action defining the quotient.  It annihilates the tangent direction
to that action,
so it descends to a smooth complex-valued symmetric $2$-tensor on $P$.  The
principal $S^1$-action on $P$ is
$[u,v]\cdot z=[zu,v]$ for $z\in S^1$.  If $V$ denotes its fundamental vector
field, then $\Psi(V,\cdot)=0$ and $\Lie_V\Psi=2i\Psi$.
We next modify $\Psi$ so that its restriction to each zero-curvature plane
has the required form.  At
$(p_1,p_2)\in B$, set $x=\langle p_1,p_2\rangle$.  Define
$\mathcal T_{12}:T_{p_1}S^2\to T_{p_2}S^2$ and
$\mathcal T_{21}:T_{p_2}S^2\to T_{p_1}S^2$ by
$\mathcal T_{12}w=xw-\langle w,p_2\rangle p_1$ and
$\mathcal T_{21}w=xw-\langle w,p_1\rangle p_2$, respectively.
These are smooth, polynomial, diagonally $\mathrm{SO}(3)$-equivariant bundle
maps, including when $p_1=\pm p_2$.  On $TB$, define
$\mathscr K(w_1,w_2)=(\mathcal T_{21}w_2,\mathcal T_{12}w_1)$ and
$\mathscr J_0(w_1,w_2)=(-\mathcal T_{21}w_2,\mathcal T_{12}w_1)$.
On the zero-curvature plane $\Pi_b$, one has
$\mathscr K E_b=F_b$, $\mathscr K F_b=E_b$,
$\mathscr J_0E_b=F_b$, and $\mathscr J_0F_b=-E_b$.  Put
$c_t=(1+2t)^{-1}$, $m_+=E_b+F_b$, and $m_-=F_b-E_b$.  The metric $g_t$
satisfies $|m_+|_{g_t}^2=c_t$, $|m_-|_{g_t}^2=1$, and
$\langle m_+,m_-\rangle_{g_t}=0$.
Define the endomorphism
$\mathscr D_t=\tfrac12(\sqrt{c_t}+c_t^{-1/2})\Id
+\tfrac12(\sqrt{c_t}-c_t^{-1/2})\mathscr K$, and let
$\mathscr J_t=\mathscr J_0\mathscr D_t$.  The corresponding ordered
$g_t$-orthonormal basis of $\Pi_b$ is
$e_1=c_t^{-1/2}m_+$, $e_2=m_-$; denote its dual coframe by
$e^1,e^2$.  In this basis, $\mathscr J_t e_1=e_2$ and
$\mathscr J_t e_2=-e_1$.
When a base vector is used as an argument of a tensor on $P$, we mean its
$\theta_t$-horizontal lift and suppress the lift symbol.  We likewise regard
a base covector as a covector on the horizontal distribution and extend it by
zero on $V$.  Whenever one of the base endomorphisms acts on $TP$, we use its
horizontal lift through $\dd\pi|_{\ker\theta_t}$ and extend it by zero on the
vertical line $\R V$; the same symbol denotes this extension.
Define
\[
 \begin{aligned}
 \widehat\Psi_t(X,Y)&=\Psi(X,Y)-\frac{t}{2(1+t)}
 \bigl(\Psi(\mathscr KX,Y)+\Psi(X,\mathscr KY)\bigr),\\
 \mathcal Q_t(X,Y)&=\widehat\Psi_t(X,Y)+\frac i2
 \bigl(\widehat\Psi_t(\mathscr J_tX,Y)
       +\widehat\Psi_t(X,\mathscr J_tY)\bigr).
 \end{aligned}
\]
The tensor $\Psi$ is already defined on $P$, while $\mathscr K$ and
$\mathscr J_t$ are base endomorphisms whose horizontal lifts are globally
defined.  Hence $\mathcal Q_t$ is a smooth complex-valued symmetric $2$-tensor
on all of $P$.  Because the lifted endomorphisms vanish on $V$ and are
invariant under the principal circle action, $\mathcal Q_t$ is horizontal
and satisfies $\Lie_V\mathcal Q_t=2i\mathcal Q_t$.
\begin{proposition}                                                 \label{prop:tensor-on-zero-curvature-planes}
For every $\widetilde\Pi\in\tZ$ lying over $\Pi_b$, including those based at
points with $p_1=\pm p_2$,
$\mathcal Q_t|_{\widetilde\Pi}=\delta(e^1-ie^2)\otimes(e^1-ie^2)$ for a complex number
$\delta\neq0$.  The quantity $|\delta|$ has a
positive uniform lower bound on $\tZ$.  Moreover, $\mathcal Q_t$ is
parallel for the induced flat metric on every local horizontal lift of
$T_b$.  More precisely, if $\Sigma_b$ is such a lift, then
$\nabla^{\Sigma_b}(\mathcal Q_t|_{\Sigma_b})=0$, where
$\nabla^{\Sigma_b}$ is the intrinsic covariant derivative.
\end{proposition}

\begin{proof}
Let $\sigma=\Psi(E_b,F_b)$.
Let $\widetilde E_b$ and $\widetilde F_b$ be their respective lifts to the two
$S^3$ factors that are horizontal for the Hopf connections.  These lifts are
nonzero.  Since the kernel of each $\zeta_i$ is the vertical direction of the
Hopf fibration,
$\zeta_1(\widetilde E_b)\neq0$ and $\zeta_2(\widetilde F_b)\neq0$.  The
vectors $E_b$ and $F_b$ lie in different sphere factors, so
$\sigma=\tfrac12\zeta_1(\widetilde E_b)\zeta_2(\widetilde F_b)\neq0$.
In the basis $E_b,F_b$, the matrix of $\Psi|_{\widetilde\Pi}$ is
\[
 \begin{pmatrix}0&\sigma\\ \sigma&0\end{pmatrix},
\]
whereas the symmetric form
$\tfrac12\{\Psi(\mathscr K\,\cdot,\cdot)+\Psi(\cdot,\mathscr K\,\cdot)\}$
has matrix $\sigma\Id$.  After changing to the orthonormal
basis $e_1,e_2$, one obtains
\[
 \widehat\Psi_t|_{\widetilde\Pi}=\delta\begin{pmatrix}1&0\\0&-1\end{pmatrix}.
\]
Here $\delta=2(1+2t)\sigma/(1+t)\neq0$.
The defining identities for $\mathscr J_t$ show that the symmetric form
$\tfrac12\{\widehat\Psi_t(\mathscr J_t\,\cdot,\cdot)
+\widehat\Psi_t(\cdot,\mathscr J_t\,\cdot)\}$ has matrix
$-\delta\begin{psmallmatrix}0&1\\1&0\end{psmallmatrix}$.
This proves the formula in the proposition.  The bundle maps $\mathcal T_{12}$,
$\mathcal T_{21}$, $\mathscr K$, and $\mathscr J_t$ are smooth at
$p_1=\pm p_2$, and the calculation uses no division by
$|p_1\times p_2|$; hence the same formula applies there.  The
continuous positive function $|\delta|$ has a positive lower bound on the
compact set $\tZ$.

It remains to prove that the restriction is parallel.  With the Hopf
convention fixed above, the
connection chosen in Section~\ref{sec:connection} satisfies
$\theta_0=\eta_1+\eta_2$ and
$\theta_t=\theta_0+\tfrac t2\pi^*\alpha$, and it has curvature
$\dd\theta_t=\pi^*\Omega_t$.  Indeed, $\theta_0$ annihilates the infinitesimal
generator $(iu,-iv)$ of the circle action defining the quotient, evaluates to
one on the generator of the principal circle action represented by $(iu,0)$,
and therefore descends to the connection form on $P$ normalized by
$\theta_0(V)=1$.  Since $\alpha|_{T_b}=0$, the
$\theta_t$-horizontal lifts of vectors tangent to $T_b$ coincide with
their $\theta_0$-horizontal lifts.  Moreover,
$\Omega_t|_{T_b}=0$ implies
$\theta_t([X^{\mathrm H},Y^{\mathrm H}])=-\Omega_t(X,Y)=0$ for tangent fields $X,Y$ on
$T_b$.  Thus the horizontal distribution restricted over $T_b$ is involutive
and admits local integral surfaces.

By diagonal $\mathrm{SO}(3)$-equivariance, take $b=e_3$ and use the equator
$p(s)=(\cos s,\sin s,0)$ with its lift horizontal for the Hopf connection,
$\widetilde p(s)=2^{-1/2}(e^{-is/2},e^{is/2})$.
A direct computation gives
$\eta(\dot{\widetilde p})=0$ and
$\zeta_{\widetilde p}(\dot{\widetilde p})=i/2$.
Hence $\sigma$, and therefore $\delta$, is constant in the two product
coordinates on a local horizontal lift.  In those
coordinates the restriction of $g_t$ to $T_b$ has the constant matrix
\[
 \frac14
 \begin{pmatrix}
  1+c_t&c_t-1\\
  c_t-1&1+c_t
 \end{pmatrix}.
\]
Thus the coordinate fields, and hence $e_1,e_2$, are parallel.  The product
lift $(\widetilde p(s_1),\widetilde p(s_2))$ is a local integral surface of
the horizontal distribution for $b=e_3$.  For general $b$, choose an element
of $\mathrm{SU}(2)$ covering a
rotation that sends $e_3$ to $b$.  The $\mathrm{SU}(2)$-invariance of
$\zeta$, the $\mathrm{SO}(3)$-equivariance of $\mathscr K,\mathscr J_t$, and the
invariance of $\theta_t$ show that the same identities hold on the
corresponding local lift.  The formula in the proposition then shows that the restriction of
$\mathcal Q_t$ is parallel.
\end{proof}

On every plane $\widetilde\Pi\in\tZ$, the real part of the formula in the proposition is a
nonzero trace-free symmetric form with eigenvalues
$\pm|\delta|$.
This proves part~(3) of Theorem~\ref{thm:construction}.

%% file: sections/perturbation.tex
\section{Curvature variation along the zero-curvature planes}                      \label{sec:perturbation}

Set $h=\operatorname{Re}\mathcal Q_t$ and $G_s=G_{t,\eps}+s h$.
The tensor $h$ is horizontal, so $h(U,\cdot)=0$.  For sufficiently small
$|s|$, the tensor $G_s$ is a Riemannian metric.
Throughout this section, $\nabla=\nabla^{G_{t,\eps}}$, and unmarked inner
products and norms are taken with respect to $G_{t,\eps}$.

Fix $\widetilde\Pi\in\tZ$ with $\widetilde\Pi\subset T_pP$, and let
$\Sigma_b$ be a local horizontal lift through $p$ such that
$T_p\Sigma_b=\widetilde\Pi$.  By
$\Omega_t|_{T_b}=0$, \eqref{eq:connection-koszul}, and the
total geodesy of $T_b$, the surface $\Sigma_b$ is flat and totally
geodesic for $G_{t,\eps}$.
Proposition~\ref{prop:tensor-on-zero-curvature-planes}
shows that the metrics induced by $G_s$ on $\Sigma_b$ have constant
coefficients in a parallel frame and are flat for all sufficiently small
$|s|$.

Let $\II_s(X,Y)=(\nabla_X^{G_s}Y)^{\perp_s}$ be the second fundamental form
of this fixed surface in $(P,G_s)$, and let
$\dot\II=\left.\partial_s\II_s\right|_{s=0}$.
For any unit normal $N$, write
$\dot\II^N=\langle\dot\II,N\rangle$.
Since $\II_0=0$, differentiating the Gauss equation identifies the first
variation of ambient sectional curvature with the first variation of the
intrinsic Gaussian curvature, as in \cite[Lemma~4.1]{Strake1987}.  The
induced metrics are flat, so both variations vanish pointwise.  The
second-order contribution below comes from differentiation along the
$S^1$-orbits.

For $X,Y\in T\Sigma_b$, the variation formula for the Levi-Civita
connection gives
\[
 2\langle\dot\II(X,Y),U\rangle
 =(\nabla_X h)(Y,U)+(\nabla_Y h)(X,U)-(\nabla_U h)(X,Y).
\]
Use $S^1$-invariant extensions, so that $[U,X]=[U,Y]=0$.  Since
$h(U,\cdot)=0$, the equivariance relation $\Lie_V\mathcal Q_t=2i\mathcal Q_t$
and $U=\eps^{-1}V$ imply
$\dot\II^U=-\tfrac12\Lie_U h
=-\eps^{-1}\operatorname{Re}(i\mathcal Q_t)$.
It follows from Proposition~\ref{prop:tensor-on-zero-curvature-planes} that
$\dot\II^U$ is parallel and trace-free on $\Sigma_b$, with
eigenvalues $\pm|\delta|/\eps$.
If $N$ is a horizontal unit normal to $\Sigma_b$, then
\[
 2\dot\II^N(X,Y)
 =(\nabla_X h)(Y,N)+(\nabla_Y h)(X,N)-(\nabla_N h)(X,Y).
\]
Here $X$, $Y$, and $N$ are horizontal.
The connection formulas \eqref{eq:connection-koszul} and the smoothness of
$h$ on the fixed principal bundle therefore yield
$|\dot\II^N|\leq C$ independently of small $\eps$.

The intrinsic curvature of $\Sigma_b$ is zero for every $s$.
For a $G_s$-orthonormal normal frame $\{N_\alpha(s)\}$, write
$\II_s^\alpha=\langle\II_s,N_\alpha(s)\rangle_{G_s}$, and at $s=0$
write $\dot\II^\alpha
=\langle\dot\II,N_\alpha(0)\rangle$.
Let $e_1,e_2$ be the fixed $G_{t,\eps}$-orthonormal basis of the
zero-curvature plane, set
$\dot\II_{ij}^\alpha=\dot\II^\alpha(e_i,e_j)$, and compute determinants of
restricted metrics in this basis.
With our curvature convention, the Gauss equation gives
\[
 \sec_{G_s}(T_p\Sigma_b)
 =\frac{\sum_\alpha\{\II_s^\alpha(e_1,e_2)^2
 -\II_s^\alpha(e_1,e_1)\II_s^\alpha(e_2,e_2)\}}
 {\det(G_s|_{\widetilde\Pi})}.
\]
Since $\II_0=0$, one has
$\II_s^\alpha(e_i,e_j)=s\dot\II_{ij}^\alpha+O(s^2)$ and
$\det(G_s|_{\widetilde\Pi})=1+O(s)$.  Set
$a_\eps(\widetilde\Pi):=\sum_\alpha((\dot\II_{12}^\alpha)^2
-\dot\II_{11}^\alpha\dot\II_{22}^\alpha)$.  The eigenvalue formula for
$\dot\II^U$ and the uniform bounds for $\dot\II^N$ with horizontal $N$ give
\begin{equation}                                                                    \label{eq:a-epsilon-lower}
 \sec_{G_s}(T_p\Sigma_b)=s^2a_\eps(\widetilde\Pi)+O(s^3)
 \;\text{and}\; a_\eps(\widetilde\Pi)\geq\frac{\kappa}{\eps^2}-C_0.
\end{equation}
Here $\kappa>0$ and $C_0>0$ are uniform on $\tZ$.  Indeed, the trace-free form
$\dot\II^U$ has determinant $-|\delta|^2/\eps^2$, while the contributions from
horizontal normal vectors are uniformly bounded.

\section{Uniform control of nearby planes}                                         \label{sec:nearby-planes}

Let $K(s,\Pi)=\sec_{G_s}(\Pi)$ on the Grassmann bundle, equipped with
the $\eps$-independent auxiliary metric fixed in
Section~\ref{sec:connection}.  Use the tubular coordinates
$(z,n)$, with
$n=(\xi,\vartheta)\in\nu^{\mathrm h}_z(\tZ)\oplus
\nu^{\mathrm v}_z(\tZ)$ in the fixed normal splitting from
Section~\ref{sec:connection}.  For every $z\in\tZ$, one has
$K(0,z,0)=0$ and $\dd_n K(0,z,0)=0$; the second identity holds because
$\tZ$ lies in the critical set, as shown in Section~\ref{sec:connection}.
The first variation calculation above also gives
$\partial_s K(0,z,0)=0$.  Set
$r_\eps(z)=\dd_n(\partial_s K)(0,z,0)$.  Taylor expansion now gives
\begin{equation}                                                                    \label{eq:tubular-expansion}
 K(s,z,n)
 = \frac12H_\eps(z)(n,n)
 +s\,r_\eps(z)(n)
 +s^2a_\eps(z) + O\bigl(|n|^3+|s||n|^2+s^2|n|+|s|^3\bigr).
\end{equation}
Here $H_\eps$ is the restriction of the Hessian to $\nu(\tZ)$, as defined in
Section~\ref{sec:connection}, and $|n|$ is the fixed auxiliary norm.  The
implicit constants are uniform in $z\in\tZ$ after $t$ and $\eps$ have been
fixed; no uniform remainder bound as $\eps\to0$ is used.  Split the normal
covector
$r_\eps=(r_{\eps,\mathrm h},r_{\eps,\mathrm v})$ according to the fixed orthogonal splitting
$\nu(\tZ)=\nu^{\mathrm h}(\tZ)\oplus\nu^{\mathrm v}(\tZ)$.

\begin{lemma}                                                                       \label{lem:first-variation-estimate}
For fixed sufficiently small $t$, uniformly on $\tZ$, one has
$|r_{\eps,\mathrm h}|\leq C$ and $|r_{\eps,\mathrm v}|\leq C\eps^2$.
\end{lemma}

The proof, including estimates for the first three covariant derivatives of
$\mathcal Q_t$, the linearized curvature formula, and the denominator in the
definition of sectional curvature, is given in
Appendix~\ref{app:tensor-derivatives}.

The lower bound for the normal Hessian gives the following estimate for its
inverse, with the different powers of $\eps$ in the horizontal and vertical
normal directions retained.

\begin{lemma}                                                                       \label{lem:block-inverse}
Let $E=E_{\mathrm h}\oplus E_{\mathrm v}$ be an orthogonal Euclidean splitting,
and let $H_\eps$ be a positive definite symmetric bilinear form on $E$, also
viewed as an isomorphism from $E$ to $E^*$.  Suppose, for $x\in E_{\mathrm h}$
and $y\in E_{\mathrm v}$, that
$H_\eps(x+y,x+y)\geq\mu(|x|^2+\eps^4|y|^2)$.  Then every
$r=(r_{\mathrm h},r_{\mathrm v})\in E^*$ satisfies
\begin{equation}                                                                    \label{eq:block-inverse}
 \langle H_\eps^{-1}r,r\rangle
 \leq\mu^{-1}\bigl(|r_{\mathrm h}|^2+\eps^{-4}|r_{\mathrm v}|^2\bigr).
\end{equation}
The splitting need not be orthogonal with respect to $H_\eps$.
\end{lemma}

\begin{proof}
Let $D_\eps=\operatorname{diag}(\Id_{E_{\mathrm h}},\eps^2\Id_{E_{\mathrm v}})$.  The
hypothesis says that
$D_\eps^{-1}H_\eps D_\eps^{-1}\geq\mu\Id$.  Hence
$H_\eps^{-1}
\leq\mu^{-1}D_\eps^{-2}$, which is \eqref{eq:block-inverse}.
\end{proof}

The normal Hessian estimate \eqref{eq:normal-hessian-lower-bound} and
Lemmas~\ref{lem:first-variation-estimate} and~\ref{lem:block-inverse} imply
pointwise, for $z\in\tZ$,
\[
 \langle H_\eps(z)^{-1}r_\eps(z),r_\eps(z)\rangle
 \leq C\bigl(|r_{\eps,\mathrm h}(z)|^2+\eps^{-4}|r_{\eps,\mathrm v}(z)|^2\bigr)\leq C_1.
\]
Together with \eqref{eq:a-epsilon-lower}, this yields
\begin{equation}                                                                    \label{eq:reduced-coefficient}
 a_\eps(z)
 -\frac12\langle H_\eps(z)^{-1}r_\eps(z),r_\eps(z)\rangle
 \geq\frac{\kappa}{\eps^2}-C_2.
\end{equation}
The constants in the following minimization lemma are uniform over the
compact submanifold $\tZ$.  The same minimization in the normal directions
is used by Brendle and Hung; compare
\cite[Lemma~2.1]{BrendleHung2026}.  We state the
vector-bundle version needed here.  In its application below, the preceding
estimates retain the different $\eps$-dependence of the horizontal and
vertical normal directions.

\begin{lemma}                                                                       \label{lem:tubular-minimization}
Let $E\to Z$ be a Euclidean vector bundle over a compact manifold $Z$, and let
$F(s,z,n)$ be smooth near the zero section of $\R\times E$.  Suppose
$F(0,z,0)=0$, $\dd_n F(0,z,0)=0$, and $\partial_s F(0,z,0)=0$.  Assume that the Hessian
$H_z=\dd_n^2F(0,z,0)$ is positive definite, and
$a_z-\frac12\langle H_z^{-1}r_z,r_z\rangle\geq\gamma>0$,
where $r_z=\dd_n(\partial_s F)(0,z,0)$ and
$a_z=\tfrac12\partial_s^2 F(0,z,0)$.  Then there are $r_0>0$ and $s_0>0$ such that
$F(s,z,n)>0$ for all $z\in Z$, $|n|<r_0$, and $0<|s|<s_0$.
\end{lemma}

The lemma is proved in Appendix~\ref{app:tubular-minimization}.

Fix $t\in(0,t_0)$.  By \eqref{eq:reduced-coefficient}, define
$\eps_0(t)>0$ small enough that
$\eps_0(t)^2\leq c_0t^3$ and
$\kappa/\eps^2-C_2>0$ whenever $0<\eps<\eps_0(t)$.  For every such
$\eps$, applying
Lemma~\ref{lem:tubular-minimization} to
\eqref{eq:tubular-expansion} shows that the sectional curvature of $G_s$ is
strictly positive in a tubular neighborhood $\mathcal U$ of $\tZ$ whenever
$0<|s|<s_1(t,\eps)$.  Choose a smaller tubular neighborhood
$\mathcal U'$ with $\overline{\mathcal U'}\subset\mathcal U$.  On the compact
set $\Gr_2(TP)\setminus\mathcal U'$, the background metric $G_{t,\eps}$ has a
positive minimum of sectional curvature.  Continuity therefore preserves
positivity there after decreasing $s_1(t,\eps)$ to a positive number
$s_0(t,\eps)$.  We have proved
$\sec_{G_{t,\eps}+s\operatorname{Re}\mathcal Q_t}(\Pi)>0$ for every
tangent two-plane $\Pi$ and every sufficiently small nonzero $s$.
This proves part~(4) and completes the proof of Theorem~\ref{thm:construction}.

%% file: sections/technical-appendix.tex
\section{Covariant derivatives of the complex-valued symmetric tensor}  \label{app:tensor-derivatives}

Fix $t\in(0,t_0)$.  All constants in this appendix may depend on $t$, but
they are uniform for $0<\eps<\eps_0(t)$.  Throughout the appendix,
$\nabla=\nabla^{G_{t,\eps}}$, and unmarked inner products are taken with
respect to $G_{t,\eps}$.  Over a local trivialization of $P$, choose a smooth
$g_t$-orthonormal frame
$\bar X_1,\ldots,\bar X_4$ on $B$, and let
$X_a=\bar X_a^{\mathrm H}$ be its $S^1$-invariant horizontal lifts.  Put
$U=\eps^{-1}V$.  Then $(X_1,\ldots,X_4,U)$ is
$G_{t,\eps}$-orthonormal and $[U,X_a]=0$.  Let
$\Gamma_{ab}^{c}$ be defined by
$\nabla_{\bar X_a}^{g_t}\bar X_b=\Gamma_{ab}^{c}\bar X_c$, and set
$\Omega_{ab}=\Omega_t(\bar X_a,\bar X_b)$.  We use the Einstein summation
convention for repeated frame indices.  Then
\begin{equation}                                                                    \label{eq:appendix-frame-connection}
 \nabla_{X_a}X_b=\Gamma_{ab}^{c}X_c-\frac\eps2\Omega_{ab}U,
 \;\text{and}\;
 \nabla_{X_a}U=\nabla_U X_a=\frac\eps2\Omega_{ac}X_c,
 \;\text{while}\;
 \nabla_U U=0.
\end{equation}
The coefficients $\Gamma_{ab}^{c}$ and $\Omega_{ab}$, together with all
derivatives in horizontal directions needed below, are bounded on a finite
collection of such charts.  They are invariant under the principal $S^1$-action.

\begin{lemma}                                                                       \label{lem:tensor-derivatives}
Let $Z_1,\ldots,Z_m,A,B$ be chosen from the invariant frame above, where
$m\leq3$.  Suppose exactly $r$ of the vectors $Z_1,\ldots,Z_m$ and exactly
$q$ of the vectors $A,B$ are equal to $U$.  Then
\begin{equation}                                                                    \label{eq:tensor-derivative-bound}
 \bigl| (\nabla^m\mathcal Q_t)(Z_1,\ldots,Z_m;A,B)\bigr|
 \leq C_m\eps^{q-r}.
\end{equation}
The same estimate holds for $h=\operatorname{Re}\mathcal Q_t$.
\end{lemma}

\begin{proof}
The semicolon in the displayed component separates the $m$ covariant
derivative slots from the two tensor slots.
The exponent $q-r$ in \eqref{eq:tensor-derivative-bound} records separately the
occurrences of $U$ among $Z_1,\ldots,Z_m$ and among $A,B$.
At order zero, $\mathcal Q_t(X_a,X_b)=O(1)$ and
$\mathcal Q_t(U,\cdot)=0$, so the assertion holds for $q=0,1,2$.
Since the $S^1$-action preserves $G_{t,\eps}$ and its Levi-Civita
connection, every $\nabla^m\mathcal Q_t$ satisfies the same equivariance
relation as $\mathcal Q_t$.  In the invariant frame, therefore,
differentiating a component in the $U$ direction is multiplication by
$2i/\eps$.  Directional derivatives of the component coefficients in the
$X_a$ directions are uniformly bounded.

The components of the covariant derivatives satisfy
\[
 \begin{aligned}
 (\nabla^{m+1}\mathcal Q_t)(Z_0,Z_1,\ldots;A,B)
 ={}&Z_0\bigl[(\nabla^m\mathcal Q_t)(Z_1,\ldots;A,B)\bigr]\\
 &\quad -\sum_j(\nabla^m\mathcal Q_t)
   (Z_1,\ldots,\nabla_{Z_0}Z_j,\ldots;A,B)\\
 &\quad -(\nabla^m\mathcal Q_t)(Z_1,\ldots;\nabla_{Z_0}A,B)\\
 &\quad -(\nabla^m\mathcal Q_t)(Z_1,\ldots;A,\nabla_{Z_0}B).
 \end{aligned}
\]
The induction also includes ordinary derivatives of the component functions
in the $X_a$ directions; they obey the same bound.  Differentiation
in the $U$ direction
increases $r$ by one and contributes $\eps^{-1}$, whereas differentiation in
an $X_a$ direction changes neither exponent.  By
\eqref{eq:appendix-frame-connection}, a connection term that replaces a
horizontal vector among $Z_1,\ldots,Z_m$ by $U$ has coefficient $O(\eps)$;
the resulting component has size $O(\eps^{q-r-1})$, so the product is
$O(\eps^{q-r})$.  Replacing an occurrence of $U$ among
$Z_1,\ldots,Z_m$ by a horizontal vector gives $O(\eps^{q-r+2})$.  Among
the last two arguments $A,B$, replacing a horizontal vector by $U$ gives
$O(\eps^{q-r+2})$, whereas replacing $U$ by a horizontal vector gives
$O(\eps^{q-r})$.  Connection coefficients involving only horizontal vectors
preserve the exponent.  This proves the inductive estimate and accounts for every
occurrence of $U$.

We verify explicitly the three orders of differentiation used below.  At first order, a
component with only horizontal arguments consists of a bounded directional
derivative and bounded connection terms involving horizontal vectors.  Placing
$U$ in one of the last two arguments gives, for
example, $-\mathcal Q_t(\nabla_{X_a}U,X_b)=O(\eps)$.  At second order, the
only term in a component with horizontal arguments that could grow as
$\eps\to0$ comes from the vertical part of $\nabla_{X_a}X_b$ and is
$O(\eps)\nabla_U\mathcal Q_t=O(1)$.  At third order, the new terms that could
grow as $\eps\to0$ have the forms
$O(\eps)\nabla_U\nabla_X\mathcal Q_t$,
$O(\eps)\nabla_X\nabla_U\mathcal Q_t$, and
$O(\eps^2)\nabla_U^2\mathcal Q_t$; each is $O(1)$ by the exponent just
proved.  Horizontal derivatives of the coefficients in
\eqref{eq:appendix-frame-connection} are bounded, and their vertical
derivatives vanish.  This accounts for all terms arising from the connection
through order three and proves the lemma.
\end{proof}

In particular, Lemma~\ref{lem:tensor-derivatives} shows that every component of
$\nabla^m h$ evaluated only on horizontal vectors is bounded by $C_m$ for
$0\leq m\leq3$.

\begin{proof}[Proof of Lemma~\ref{lem:first-variation-estimate}]
We first record the formula for the linearization of the curvature tensor with
the convention used in this paper.  An overdot in this proof denotes
$\left.\partial_s\right|_{s=0}$.  Let
$\dot{\nabla}=\left.\partial_s\nabla^{G_s}\right|_{s=0}$.  Then
\[
 2G_{t,\eps}(\dot{\nabla}_X Y,Z)
 =(\nabla_X h)(Y,Z)+(\nabla_Y h)(X,Z)-(\nabla_Z h)(X,Y),
\]
and, writing $\dot\Rm=\left.\partial_s\Rm_{G_s}\right|_{s=0}$,
\begin{equation}                                                                    \label{eq:linearized-curvature}
 \dot\Rm(X,Y,Z,W)
 =h(R^{G_{t,\eps}}(X,Y)Z,W)
  +\left\langle
    (\nabla_X\dot{\nabla})(Y,Z)-(\nabla_Y\dot{\nabla})(X,Z),W
   \right\rangle.
\end{equation}
Equivalently, if
$\nabla^2_{X,Y}h=\nabla_X(\nabla_Y h)-\nabla_{\nabla_X Y}h$, then
\[
 \begin{aligned}
 2\dot\Rm(X,Y,Z,W)={}&2h(R^{G_{t,\eps}}(X,Y)Z,W) +(\nabla^2_{X,Y}h-\nabla^2_{Y,X}h)(Z,W)
 +(\nabla^2_{X,Z}h)(Y,W)\\
 &\quad -(\nabla^2_{Y,Z}h)(X,W) -(\nabla^2_{X,W}h)(Y,Z)+(\nabla^2_{Y,W}h)(X,Z).
 \end{aligned}
\]
Formula~\eqref{eq:linearized-curvature} expresses a component of $\dot\Rm$
with horizontal arguments in terms of at most two covariant derivatives of
$h$.  One horizontal covariant derivative of such a component involves at
most three derivatives of $h$.
The connection formulas in \eqref{eq:appendix-frame-connection} and
compactness show that the components of the background curvature tensor with
horizontal arguments, and their first horizontal derivatives, are uniformly
bounded.  The preceding
estimates for covariant
derivatives with horizontal arguments therefore show that the corresponding
components of $\dot\Rm$, and their
first horizontal covariant derivatives, are uniformly bounded.

We next account for the denominator in the definition of sectional curvature.  For
vectors $A,B$ spanning a fixed plane, put
$N_s(A,B)=\Rm_{G_s}(A,B,B,A)$ and
$D_s(A,B)=G_s(A,A)G_s(B,B)-G_s(A,B)^2$, so that
$K(s,\Pi)=N_s/D_s$.  Fix $z\in\tZ$ based at $p\in P$, and choose the horizontal
$G_{t,\eps}$-orthonormal basis $X,Y$ and extend it to a smooth local choice
of spanning frame.  Choose $b$ so that $z$ lies over $\Pi_b$, and let
$\Sigma_b$ be the lift tangent to $z$ used in
Section~\ref{sec:perturbation}.  We thereby regard $N_s$ and $D_s$
as local functions of the plane.  Then
$N_0(z)=0$ and $D_0(z)=1$.  Since $N_0=K(0,\cdot)D_0$,
$K(0,z)=0$, and $\dd_\Pi K(0,z)=0$, one has
$\dd_\Pi N_0(z)=0$.  The
argument for the first variation in
Section~\ref{sec:perturbation} gives $\dot K(z)=0$, hence $\dot N(z)=0$.
Indeed, differentiating $K=N/D$ first with respect to $s$ gives
$\dot K=\dot N/D_0-N_0\dot D/D_0^2$.  Differentiating this identity in a
normal direction $n$ and evaluating at $z$ leaves only
$(\dd\dot N)_z(n)/D_0(z)$: every other term contains one of $N_0(z)$,
$(\dd_\Pi N_0)_z$, or $\dot N(z)$ and therefore vanishes.  Since $D_0(z)=1$,
$r_\eps(n)=(\dd\dot K)_z(n)=(\dd\dot N)_z(n)$.  In particular, this
calculation is independent of the normalization chosen for a nearby spanning
frame.

We first estimate curvature components with three tangent arguments and one
vertical argument.  The estimate will be used for both summands of the normal
bundle.  Decompose the variation of the second fundamental form from
Section~\ref{sec:perturbation} as
$\dot\II=\dot\II^U U+(\dot\II)^{\mathrm H}$, where
$(\dot\II)^{\mathrm H}$ takes values in
$\nu\Sigma_b\cap\ker\theta_t$.  The identity for $\dot\II^U$ and the uniform bound for
$(\dot\II)^{\mathrm H}$ from Section~\ref{sec:perturbation},
together with
Proposition~\ref{prop:tensor-on-zero-curvature-planes}, give
$\nabla^{\Sigma_b}\dot\II^U=0$ and $|(\dot\II)^{\mathrm H}|\leq C$.
Differentiate the Codazzi equation of the fixed surface
$\Sigma_b$ at $s=0$.  Since $\II_0=0$, all terms arising from
differentiating the normal projection, the induced connections, or the
tangent arguments are proportional to $\II_0$ and vanish.  The field $U$ remains unit
and normal for every $G_s$, since $h(U,\cdot)=0$.  The term arising from the
variation of the metric pairing is
$h(R^{G_{t,\eps}}(X,Y)Z,U)=0$.  Hence, for tangent fields $X,Y,Z$,
\[
 \dot\Rm(X,Y,Z,U)
 =\left\langle
  (\nabla_X^\perp\dot\II)(Y,Z)
  -(\nabla_Y^\perp\dot\II)(X,Z),U
 \right\rangle.
\]
Taking the $U$-component of the normal covariant derivative gives
\[
 \left\langle(\nabla_X^\perp\dot\II)(Y,Z),U\right\rangle
 =(\nabla_X^{\Sigma_b}\dot\II^U)(Y,Z)
  -\langle(\dot\II)^{\mathrm H}(Y,Z),\nabla_X^\perp U\rangle.
\]
The first term vanishes because $\dot\II^U$ is parallel.  By
\eqref{eq:appendix-frame-connection}, the second has absolute value at most
$C\eps$.  Therefore
\[
 |\dot\Rm(X,Y,Z,U)|\leq C\eps.
\]

By construction, a vector $n\in\nu^{\mathrm h}_z(\tZ)$ is represented by a curve of
horizontal planes lifted from $\Gr_2(TB)$.  Let $\dot p\in T_pP$ be the
velocity of the corresponding curve of footpoints in $P$; our tubular coordinates
make $\dot p$ horizontal.  Differentiating the numerator in the direction
$n$ gives the four terms obtained by replacing in turn one argument of
$\dot\Rm(X,Y,Y,X)$ by the corresponding horizontal frame variation.  It
also gives $(\nabla_{\dot p}\dot\Rm)(X,Y,Y,X)$ when the footpoint varies.
The horizontal parts of all these terms
are bounded by the curvature estimate for horizontal arguments above.  Because
the spanning frame remains horizontal along the curve,
\eqref{eq:appendix-frame-connection} shows that the vertical part of each
frame variation is $O(1)V=O(\eps)U$.  The estimate
$|\dot\Rm(X,Y,Z,U)|\leq C\eps$ therefore bounds all these terms.
The linear maps from the chosen normal coordinates to these frame variations
are uniformly bounded on the compact set $\tZ$, so the identity $K=N/D$ gives
$|r_{\eps,\mathrm h}|\leq C$.

Finally, consider the vertical graph coordinates from
Section~\ref{sec:connection}:
$\Span\{X+\vartheta_1V,Y+\vartheta_2V\}$, where $V=\eps U$.
These variations fix the base point, so only the arguments of $\dot\Rm$
vary.  Since $r_\eps=(\dd\dot N)_z$, differentiation at
$\vartheta=0$ gives
\[
 \begin{aligned}
 r_{\eps,\mathrm v}(\partial_{\vartheta_1})
 &=\dot\Rm(V,Y,Y,X)+\dot\Rm(X,Y,Y,V)
   =2\eps\dot\Rm(X,Y,Y,U),\\
 r_{\eps,\mathrm v}(\partial_{\vartheta_2})
 &=\dot\Rm(X,V,Y,X)+\dot\Rm(X,Y,V,X)
   =2\eps\dot\Rm(X,Y,U,X).
 \end{aligned}
\]
Here we used the algebraic curvature symmetries, which are preserved under
differentiation with respect to $s$.  In particular,
$\dot\Rm(X,Y,U,X)=-\dot\Rm(X,Y,X,U)$, so
$|\dot\Rm(X,Y,Z,U)|\leq C\eps$ bounds both displayed components by
$C\eps^2$.

Because $X,Y$ are horizontal, their $G_{t,\eps}$-orthonormality implies that
they are also orthonormal for $G_{t,1}$; moreover, $V$ is unit for $G_{t,1}$.
Under the identification
$\nu^{\mathrm v}_z(\tZ)=\operatorname{Hom}(\widetilde\Pi,\R V_p)$,
$\partial_{\vartheta_1}$ and $\partial_{\vartheta_2}$ correspond to the maps
that send, respectively, $X$ and $Y$ to $V$ and vanish on the other basis
vector.  They are therefore orthonormal for the fixed norm on
$\nu^{\mathrm v}_z(\tZ)$.  The two component estimates now give
$|r_{\eps,\mathrm v}|\leq C\eps^2$, completing the proof.
\end{proof}

\section{Proof of Lemma~\ref{lem:tubular-minimization}}                             \label{app:tubular-minimization}

We give the compactness argument for the minimization asserted in
Lemma~\ref{lem:tubular-minimization}.

\begin{proof}
Compactness gives $m:=\min_{z\in Z}\lambda_{\min}(H_z)>0$.
Shrink the neighborhood of the zero section and the interval in the
$s$-variable so that
$\dd_n^2F\geq m/2$.  The implicit function theorem, applied fiberwise and
uniformly over $Z$, produces a unique smooth section $n_*(s,z)$ satisfying
$\dd_n F(s,z,n_*(s,z))=0$, $n_*(0,z)=0$, and
$\partial_s n_*(0,z)=-H_z^{-1}r_z$.  Set
$\varphi(s,z)=F(s,z,n_*(s,z))$.  Differentiating gives
$\varphi(0,z)=\partial_s\varphi(0,z)=0$ and
$\tfrac12\partial_s^2\varphi(0,z)=
a_z-\tfrac12\langle H_z^{-1}r_z,r_z\rangle$.
Uniform Taylor expansion and the strict inequality in
Lemma~\ref{lem:tubular-minimization} yield
$\varphi(s,z)\geq\tfrac\gamma2s^2$ for small nonzero $s$.  Strong convexity
on each fiber then gives
$F(s,z,n)\geq\varphi(s,z)+\tfrac m4|n-n_*(s,z)|^2$,
which proves the claim.
\end{proof}

%% file: sections/topology.tex
\section{Topology of the circle bundle}                                            \label{sec:topology}

We complete the argument by identifying the circle bundle used in the metric
construction.  In the quotient model of
Section~\ref{sec:perturbation-tensor}, the character
$\chi:T^2\to S^1$, $\chi(z_1,z_2)=z_1z_2$, has as its kernel the circle
subgroup used to define $P$ and therefore induces an isomorphism
$T^2/\ker\chi\cong S^1$.  Since the two Hopf fibrations have first
Chern classes $(1,0)$ and $(0,1)$, respectively, the resulting principal
$S^1$-bundle has $c_1(P)=(1,1)$.
It is therefore the bundle used in Section~\ref{sec:connection}.

\begin{lemma}                                                                       \label{lem:topology}
The total space of the principal $S^1$-bundle of Chern class $(1,1)$ over
$S^2\times S^2$ is diffeomorphic to $S^2\times S^3$.
\end{lemma}

\begin{proof}
The long exact homotopy sequence of
$S^1\to P\to S^2\times S^2$ contains the connecting map
$\pi_2(S^2\times S^2)\cong\Z^2\to\pi_1(S^1)\cong\Z$, given by
$(m,n)\mapsto m+n$.
It is surjective because the first Chern class is primitive.  Hence
$\pi_1(P)=0$.

The Gysin sequence shows that
\[
 H^3(P;\Z)\cong\ker\bigl(H^2(S^2\times S^2;\Z)\xrightarrow{\smile(1,1)}H^4(S^2\times S^2;\Z)\bigr)\cong\Z.
\]
By Poincar\'e duality, $H_2(P;\Z)\cong\Z$, in particular with no torsion.
The vertical tangent line is trivial, so
$TP\cong\pi^*T(S^2\times S^2)\oplus\R$.
Each oriented two-sphere is spin, and therefore
$w_2(P)=\pi^*w_2(S^2\times S^2)=0$.
The Smale--Barden classification of simply connected five-manifolds implies
that a closed
simply connected spin five-manifold with torsion-free second homology is a
connected sum of copies of $S^2\times S^3$, with the number of summands
equal to the rank of $H_2$~\cite{Smale1962,Barden1965}.  The preceding
homology and spin calculations therefore give
$P\cong S^2\times S^3$.
\end{proof}